\documentclass[11pt]{amsart}
\usepackage{hyperref}
\usepackage{amsmath,amsfonts,amssymb,amscd,amsthm,amsbsy,amsxtra,bbm,bm, epsf,calc,comment,appendix}
\usepackage{color}
\usepackage{datetime}
\usepackage{latexsym}
\usepackage[english]{babel}
\usepackage{enumerate}
\usepackage{graphicx}
\usepackage{epsfig}

\def\e{\epsilon}

\def\II{{\rm I\kern-0.5exI}}
\def\III{{\rm I\kern-0.5exI\kern-0.5exI}}

\newcommand{\norm}[1]{\lVert #1 \rVert}

\newcommand{\RR}{\mathbb{R}}

\newcommand{\ZZ}{\mathbb{Z}}

\DeclareMathOperator{\loc}{\textup{loc}}

\DeclareSymbolFont{bbold}{U}{bbold}{m}{n}
\DeclareSymbolFontAlphabet{\mathbbold}{bbold}

\newcommand{\cX}{\mathcal{X}}
\newcommand{\cY}{\mathcal{Y}}

\newcommand{\tilq}{\tilde{q}}
\newcommand{\bq}{\bar{q}}
\newcommand{\brho}{\bar{\rho}}
\newcommand{\bmu}{\bar{\mu}}

\newcommand{\vp}{\varphi}

\numberwithin{equation}{section}
\newtheorem{theorem}{Theorem}[section]
\newtheorem{lemma}[theorem]{Lemma}
\newtheorem{prop}[theorem]{Proposition}
\newtheorem{cor}[theorem]{Corollary}

\theoremstyle{remark}

\theoremstyle{definition}
\newtheorem{definition}[theorem]{Definition}
\author{Matt Jacobs}
\title{Existence of solutions to reaction cross diffusion systems }
\begin{document}
\maketitle

\begin{abstract}
Reaction cross diffusion systems are a two species generalization of the porous media equation.  These systems play an important role in the mechanical modelling of living tissues and tumor growth. Due to their mixed parabolic-hyperbolic structure, even proving the existence of solutions to these equations is challenging.  In this paper, we exploit the parabolic structure of the system to prove the strong compactness of the pressure gradient in $L^2$. The key ingredient is the energy dissipation relation, which along with some compensated compactness arguments, allows us to upgrade weak convergence to strong convergence.    As a consequence of the pressure compactness, we are able to prove the existence of solutions in a very general setting and pass to the Hele-Shaw/incompressible limit in any dimension. 
\end{abstract}

\section{Introduction}
In this paper, we consider the following two species reaction cross diffusion system
\begin{equation}\label{eq:system}
\begin{cases}
\partial_t \rho_1-\nabla \cdot (\rho_1(\nabla p-V))=\rho_1 F_{1,1}(p,n)+\rho_2 F_{1,2}(p,n),\\
\partial_t \rho_2-\nabla \cdot (\rho_2(\nabla p-V))=\rho_1 F_{2,1}(p,n)+\rho_2 F_{2,2}(p,n),\\
\rho p=z(\rho)+z^*(p),\\
\partial_t n-\alpha \Delta n=-n(c_1\rho_1+c_2\rho_2),
\end{cases}
\end{equation}
on the spacetime domain $Q_{\infty}:=[0,\infty)\times \RR^d$.  The study of these systems has become extremely important in the  modelling of tissue growth and cancer \cite{BYRNE2003567, Preziosi2008, Ranft20863}  and has drawn substantial interest from the mathematical community \cite{pqv2014, perthame_2015, gwiazda_perthame, kim_tong, price_xu, bubba_perthame, Jacobs2021, alexander_kim_yao, santambrogio_dendritic}. The equations models the growth and death of two populations of cells whose densities are given by $\rho_1, \rho_2$.  The densities are linked through a convex energy $z$ (and its convex dual $z^*$), which opposes the concentration of the total density $\rho=\rho_1+\rho_2$.  The energy induces a pressure function $p$, which dissipates energy by pushing the densities down $\nabla p$. In addition, the densities flow along an external vector field $V$. The source terms that control the growth/death of the two populations depend on both the pressure and a nutrient variable $n$.  The nutrient evolves through a coupled equation that accounts  for both diffusion and consumption.  

Throughout the paper, we assume that $V\in L^2_{\loc}([0,\infty);L^2(\RR^d))$ and $\nabla \cdot V\in L^{\infty}(Q_{\infty})$.  We will also have the following assumptions on the energy $z$: 
\begin{enumerate}[(z1)]
    \item $z:\RR\to \RR\cup\{+\infty\}$ is proper, lower semicontinuous, and convex,
    \item $z(a)=+\infty$ if $a<0$ and $z(0)=0$,
    \item   there exists $a>0$ such that $z$ is differentiable at $a$ and $\sup\partial z(0)<z'(a)$,
\end{enumerate}
as well as the following assumptions on the source terms:
\begin{enumerate}[(F1)]
    \item the $F_{i,j}$ are continuous on $\RR\times [0,\infty)$ and uniformly bounded,
    \item the cross terms $F_{1,2}, F_{2,1}$ are nonnegative.
\end{enumerate}
In certain cases, we will need the additional assumption:
\begin{enumerate}[(F3)]
    \item for $n$ fixed,  $p\mapsto (F_{1,1}(p,n)+F_{2,1}(p,n))$ and $p\mapsto  (F_{1,2}(p,n)+F_{2,2}(p,n))$ are decreasing.
\end{enumerate}

Constructing weak solutions to the system (\ref{eq:system}) is challenging due to the highest order nonlinear terms $\rho_1\nabla p, \rho_2\nabla p$.  Given a sequence of approximate solutions,  one needs either strong convergence of the densities or of the pressure gradient to pass to the limit.  Due to the hyperbolic character of the first two equations, the regularity of the individual densities need not improve over time. Furthermore, it is not clear if densities with BV initial data will remain BV in dimensions $d>1$ (see \cite{carrillo_1d} and \cite{bubba_perthame} for results in one dimension).  On the other hand, summing the first two equations, one sees that the pressure $p$ and the \emph{total} density  $\rho$ satisfy the parabolic equation
\begin{equation}\label{eq:parabolic_intro}
\partial_t \rho-\nabla \cdot (\rho(\nabla p-V))=\rho_1 \big(F_{1,1}(p,n)+F_{2,1}(p,n)\big)+\rho_2\big( F_{1,2}(p,n)+F_{2,2}(p,n)\big),
\end{equation}
(note (\ref{eq:parabolic_intro}) needs to be coupled with the duality relation $\rho p=z(\rho)+z^*(p)$ in order to fully appreciate the parabolic structure).  Hence, attacking the problem through the pressure appears to be more promising. 

Indeed, recently, several authors have been able to construct solutions to certain cases of (\ref{eq:system}) by exploiting (\ref{eq:parabolic_intro}) to obtain strong convergence of the pressure gradient \cite{gwiazda_perthame, price_xu}.  The strategy of these approaches is to use the parabolic structure to obtain a priori estimates on the pressure that are strong enough to guarantee compactness. In particular, following these approaches, one tries to bound the pressure Laplacian in at least $L^1$ and then obtain some additional (arbitrarily weak) time regularity.  As it turns out, both space and time regularity can be problematic.  It is not clear whether spatial regularity can hold without some structural assumptions on the sources terms $F_{i,j}$ or in the presence of a non-zero vector field $V$.  Time regularity also becomes problematic in the (important) special case where the energy $z$ enforces the incompressibility constraint $\rho\leq 1$.  Indeed, in the incompressible case, the coupling between the total density $\rho$ and the pressure $p$ is degenerate and it is not clear how to convert time regularity for $\rho$ (easy) into time regularity for $p$ (hard).

In this paper, rather than establish the strong convergence of the pressure gradient through regularity, we instead prove it directly by exploiting the energy dissipation relation associated to (\ref{eq:parabolic_intro}).
In order to explain our strategy more fully, we need to introduce a change of variables that will make our subsequent analysis easier.  Thanks to the duality relation $\rho p=z(\rho)+z^*(p)$, the term $\rho\nabla p$ is equivalent to $\nabla z^*(p)$.  
This suggests the natural change of variables $q=z^*(p)$. 
 Since the pressure is only relevant on the set $\rho>0$, we can essentially treat $z^*$ as a strictly increasing function.  As a result, we can completely rewrite the system (\ref{eq:system}) and the parabolic equation (\ref{eq:parabolic_intro}) in terms of $q$ instead of $p$ (c.f. Section \ref{sec:equivalence} and \ref{sec:main} for the rigorous justification).  Doing so, we get the equivalent system 
\begin{equation}\label{eq:system_q}
\begin{cases}
\partial_t \rho_1-\nabla \cdot(\frac{\rho_1}{\rho}\nabla q)+\nabla \cdot (\rho_1 V)=\rho_1 F_{1,1}\big((z^*)^{-1}(q),n\big)+\rho_2 F_{1,2}\big((z^*)^{-1}(q),n\big),\\
\partial_t \rho_2-\nabla \cdot(\frac{\rho_2}{\rho}\nabla q)+\nabla \cdot (\rho_2 V)=\rho_1 F_{2,1}\big((z^*)^{-1}(q),n\big)+\rho_2 F_{2,2}\big((z^*)^{-1}(q),n\big),\\
\rho q=e(\rho)+e^*(q),\\
\partial_t n-\alpha \Delta n=-n(c_1\rho_1+c_2\rho_2),
\end{cases}
\end{equation}
where $e$ is the unique convex function such that 
\[
e(a)=\begin{cases}
az(a)-2\int_0^a z(s)\, ds & \textup{if}\;\; z(a)\neq +\infty,\\
+\infty & \textup{otherwise.}
\end{cases}
\]

It is worth noting that the change of variables from $p$ to $q$ is essentially the reverse direction of Otto's celebrated interpretation of the porous media equation as a $W^2$ gradient flow \cite{otto2001pme}. Indeed, the $p$ variable can be interpreted as a Kantorovich potential for the quadratic optimal transport distance, while the $q$ variable is instead the dual potential for an $H^{-1}$ distance. While the optimal transport interpretation of the system is more physically natural, the linearity of the $H^{-1}$ structure is advantageous for our arguments. Indeed, summing the first two equations of (\ref{eq:system_q}), we get a more linear analogue of (\ref{eq:parabolic_intro}):
\begin{equation}\label{eq:parabolic_q}
\partial_t \rho-\Delta q+ \nabla \cdot (\rho V)=\mu,
\end{equation}
where we have defined $\mu:=\rho_1 \big(F_{1,1}\big((z^*)^{-1}(q),n\big)+ F_{2,1}\big((z^*)^{-1}(q),n\big)\big)+\rho_2\big( F_{1,2}\big((z^*)^{-1}(q),n\big)+F_{2,2}\big((z^*)^{-1}(q),n\big)\big)$ for convenience.  

Now we are ready to give an outline of our strategy. As we mentioned earlier, the key idea is to exploit the energy dissipation relation associated to (\ref{eq:parabolic_q}).  Given any nonnegative test function $\omega\in W^{1,\infty}_c([0,\infty))$ depending on time only, the dissipation relation states that
\begin{equation}\label{eq:edr_intro}
\int_{Q_{\infty}} \omega|\nabla q|^2-e(\rho)\partial_t \omega+\omega e^*(q)\nabla \cdot V-\omega \mu q= \int_{\RR^d} \omega(0)e(\rho^0)
\end{equation}
where $\rho^0$ is the initial total density and we recall that $Q_{\infty}=[0,\infty)\times\RR^d$ is the full space-time domain.  Suppose we have a sequence $(\rho_k, q_k, \mu_k)$ of solutions to equation (\ref{eq:parabolic_q}) with the same initial data $\rho^0$ that converges weakly to a limit point $(\brho, \bq, \bmu)$.  Thanks to the linearity of (\ref{eq:parabolic_q}),  the limit point $(\brho, \bq, \bmu)$   will also be a solution of (\ref{eq:parabolic_q}). As a result, we can expect that both $(\rho_k, q_k, \mu_k)$ and $(\brho, \bq, \bmu)$ satisfy the dissipation relation (\ref{eq:edr_intro}). Hence, we can conclude that
\[
\int_{Q_{\infty}} \omega|\nabla q_k|^2-e(\rho_k)\partial_t \omega+\omega e^*(q_k)\nabla \cdot V-\omega \mu_k q_k=\int_{Q_{\infty}} \omega|\nabla \bq|^2-e(\brho)\partial_t \omega+\omega e^*(\bq)\nabla \cdot V-\omega \bmu \bq,
\]
If we can prove that $e(\rho_k), e^*(q_k)$ converge weakly to $e(\brho), e^*(\bq)$ respectively and 
\begin{equation}\label{eq:qmu_limit_intro}
\limsup_{k\to\infty} \int_{Q_{\infty}} \omega \mu_kq_k\leq \int_{Q_{\infty}} \omega \bmu\bq,
\end{equation}
then we have the upper semicontinuity property
\begin{equation}\label{eq:q_usc}
\limsup_{k\to\infty} \int_{Q_{\infty}} \omega |\nabla q_k|^2\leq \int_{Q_{\infty}} \omega |\nabla q|^2,
\end{equation}
which automatically implies that $\nabla q_k$ converges strongly in $L^2_{\loc}([0,\infty);L^2(\RR^d))$ to $\nabla \bq$.
Thus, the energy dissipation relation gives us a way to upgrade some weak convergence properties into strong gradient convergence.

Of course, in order to exploit this idea, we need:
\begin{enumerate}[(i)]
    \item enough regularity to ensure that the dissipation relation (\ref{eq:edr_intro}) is valid,
    \item enough compactness to prove the weak convergence of the energies $e(\rho_k), e^*(q_k)$,
    \item enough compactness to verify the nonlinear limit (\ref{eq:qmu_limit_intro}).
\end{enumerate}
The amount of a priori regularity needed for (i) is very low, thus, this point does not pose much of a problem. However, obtaining the compactness needed for points (ii) and (iii) is more delicate.  Exploiting convex duality, the weak convergence of the energies $e(\rho_k), e^*(q_k)$ is essentially equivalent to the weak convergence of the product $\rho_k q_k$ (c.f. Proposition \ref{prop:e_e^*}).  While we may not know strong convergence of either $\rho_k$ or $q_k$ separately, we can still obtain the weak convergence of the product through compensated compactness arguments (c.f. Lemma \ref{lem:spacetime_cc}).  When $e^*$ is strictly convex, the weak convergence of the energy $e^*(q_k)$ to $e^*(q)$ actually implies that $q_k$ converges to $q$ locally in measure. Thus, in this case, verifying the limit (\ref{eq:qmu_limit_intro}) becomes trivial.  When the strict convexity of $e^*$ fails, we will still be able to verify the limit (\ref{eq:qmu_limit_intro}) as long as we add the additional structural assumption (F3) on the source terms.

Once we have obtained the strong convergence of the pressure gradient, constructing solutions to the system (\ref{eq:system_q}) (and hence the system (\ref{eq:system})) is straightforward via a vanishing viscosity approach (note adding viscosity to the system is compatible with our energy dissipation based argument). Furthermore,
 the above strategy works even when the energy is allowed to change along the approximating sequence. Hence, we can also use the above arguments to show that solutions to the system (\ref{eq:system}) with the porous media energy $z_m(a)=\frac{1}{m-1}a^m$ converge to the incompressible limit system with the energy $z_{\infty}(a)=0$ if $a\in [0,1]$ and $+\infty$ otherwise.

 \subsection{Main results}
 For the reader's convenience, in this subsection, we collect some of our main results. To prevent the introduction from becoming too bloated, we shall state our results somewhat informally.  The rigorous analogues of these results can be found in Section \ref{sec:main}.
 
 Our first result concerns the case where the density-pressure coupling is non-degenerate i.e. $z$ is differentiable on $(0,\infty)$. 
 \begin{theorem}\label{thm:existence_strict}
Suppose that $z$ is an energy satisfying assumptions (z1-z3) such that $\partial z(a)$ is a singleton for all $a>0$ and suppose that the source terms satisfy assumptions (F1-F2). Given initial data $\rho_1^0, \rho_2^0, n^0$ such that $e(\rho_1^0+\rho_2^0)
\in L^1(\RR^d)$, there exists a weak solution $(\rho_1, \rho_2, p, n)$ to the system (\ref{eq:system}).
 \end{theorem}
 
 When the density-pressure coupling becomes degenerate, we need to add the additional assumption (F3) on the source terms.  
  \begin{theorem}\label{thm:existence_degenerate}
Suppose that $z$ is an energy satisfying assumptions (z1-z3) and suppose that the source terms satisfy assumptions (F1-F3). Given initial data $\rho_1^0, \rho_2^0, n^0$ such that $e(\rho_1^0+\rho_2^0)
\in L^1(\RR^d)$, there exists a weak solution $(\rho_1, \rho_2, p, n)$ to the system (\ref{eq:system}).
 \end{theorem}
 
In addition to our existence results, we also show that solutions of the system with the porous media energy $z_m(a):=\frac{1}{m-1}a^m$ converge to a solution of the system with the incompressible energy 
\[z_{\infty}(a):=
\begin{cases}
0& \textup{if} \;\;a\in [0,1] \\
+\infty & \textup{otherwise} \\
\end{cases}
\]
as $m\to\infty$.
\begin{theorem}\label{thm:incompressible_limit}
Let $\rho_1^0, \rho_2^0, n^0$ be initial data such that $\rho_1^0+\rho_2^0\leq 1$ almost everywhere. Suppose that the source terms satisfy (F1-F3). If $(\rho_{1,m}, \rho_{2,m}, p_m, n_m)$ is a sequence of solutions to the system (\ref{eq:system}) with the energy $z_m$ and the fixed initial data $(\rho_1^0, \rho_2^0,n^0)$, then there exists a limit point of the sequence $(\rho_{1,\infty},\rho_{2,\infty}, p_{\infty}, n_{\infty})$ that solves the system (\ref{eq:system}) with the incompressible energy $z_{\infty}$.
\end{theorem}
 Theorem \ref{thm:incompressible_limit} is just a special case of our more general convergence result, Theorem \ref{thm:main}, which shows that one can extract limit solutions for essentially any reasonable sequence of energies.  Nonetheless, the statement of Theorem \ref{thm:main} is a bit too complicated to be cleanly summarized in the introduction, so we leave it to be stated for the first time in Section \ref{sec:main}.

\subsection{Limitations and other directions}
Unfortunately, our approach cannot handle the more challenging case where $\rho_1, \rho_2$ have different mobilities or where $\rho_1, \rho_2$ flow along different vector fields $V_1, V_2$. These cases are known to be extremely difficult, however see \cite{Kim_Meszaros} and \cite{kim_tong} for some partial results.  When the mobilities are different, the analogue of (\ref{eq:parabolic_q}) is a nonlinear parabolic equation with potentially discontinuous coefficients.  As a result, one cannot do much with the limiting variables $\brho, \bq$. When the densities flow along different vector fields, verifying the upper semicontinuity property (\ref{eq:q_usc}) requires proving the weak convergence of the terms $\rho_{1,k} \nabla q_k$ and $\rho_{2,k}\nabla q_k$. Since this essentially requires knowing strong compactness for $\nabla q_k$ in the first place, it completely defeats the purpose of the argument.  

Nonetheless, it would be interesting to see if this strategy could be applied to other systems of equations that have some parabolic structure.  For instance, if $\{W_{i,j}\}_{i,j\in \{1,2\}}$ are convolution kernels whose symbols are dominated by $(-\Delta)^{1/2}$ i.e. $\limsup_{|\xi|\to \infty}\frac{|\hat{W}_{i,j}(\xi)|}{|\xi|}=0$, then 
it should be possible to extend our arguments to the more general system 
 \begin{equation}\label{eq:future_system_q}
\begin{cases}
\partial_t \rho_1-\nabla \cdot(\frac{\rho_1}{\rho}\nabla q)+\nabla \cdot (\rho_1 V)+W_{1,1}*\rho_1+W_{1,2}*\rho_2=\rho_1 F_{1,1}\big((z^*)^{-1}(q),n\big)+\rho_2 F_{1,2}\big((z^*)^{-1}(q),n\big),\\
\partial_t \rho_2-\nabla \cdot(\frac{\rho_2}{\rho}\nabla q)+\nabla \cdot (\rho_2 V)+W_{2,1}*\rho_1+W_{2,2}*\rho_2=\rho_1 F_{2,1}\big((z^*)^{-1}(q),n\big)+\rho_2 F_{2,2}\big((z^*)^{-1}(q),n\big),\\
\rho q=e(\rho)+e^*(q),\\
\partial_t n-\alpha \Delta n=-n(c_1\rho_1+c_2\rho_2),
\end{cases}
\end{equation}
(perhaps with some other mild requirements on the $W_{i,j}$), however, we will not pursue this line of inquiry further in this work.

\subsection{Paper outline}
The rest of the paper is organized as follows.   In Section \ref{sec:equivalence}, we explore some of the consequences of the change of variables $q=z^*(p)$.  After this Section, we will focus only on the transformed system (\ref{eq:system_q}) until Section \ref{sec:main}.  In Section \ref{sec:cc}, we provide some generic convex analysis and compensated compactness arguments needed for the weak convergence of the primal and dual energies.  In Section \ref{sec:estimates}, we analyze parabolic PDEs, establishing basic estimates and the energy dissipation relation.  Finally, in Section \ref{sec:main}, we combine our work to prove the main results of the paper.

\

\section{The transformation $q=z^*(p)$}\label{sec:equivalence}
In this section, we will explore some of the consequences of the transformation $q=z^*(p)$.  Note that the full verification of the equivalence between the systems (\ref{eq:system}) and (\ref{eq:system_q}) will not occur until the final section, Section \ref{sec:main}.  Before we begin our work in this section, let us give a bit more motivation for introducing this change of variables.  First of all, the spatial derivative in the parabolic equation (\ref{eq:parabolic_q}) is linear with respect to $q$, whereas the spatial derivative in parabolic equation for the $p$ variable (\ref{eq:parabolic_intro}) is not.  As a result, establishing the strong $L^2$ gradient compactness for $q$ is simpler than for $p$.  Furthermore, the $q$ variable is always nonnegative, while certain choices of $z$ will lead to a $p$ variable that is not bounded from below.  The lack of lower bounds on $p$ leads to some very annoying integrability issues that are completely absent when one works with $q$ instead. 

 We begin by establishing the fundamental properties of the transformation $q=z^*(p)$.  In particular, we will show that the transformation is essentially invertible.
 \begin{lemma}
If $z$ is an energy satisfying (z1-z3), then $z^*$ is nonnegative, nondecreasing, and $(z^*)^{-1}$ is well defined and Lipschitz on $z^*(\RR)\cap (0,\infty)$.
\end{lemma}
\begin{proof}
Given any $b\in \RR$, we have 
\[
z^*(b)=\sup_{a\in \RR} ab-z(a)\geq 0-z(0)=0.
\]
 It is also clear that $\inf \partial z^*(b)\geq 0$ since $z(a)=+\infty$ for any $a<0$.  If $b_1<b_2$, then $z^*(b_2)-z^*(b_1)\geq a_1(b_2-b_1)\geq 0$ where $a_1$ is any element of $\partial z^*(b_1)$. Thus, $z^*$ is both nonnegative and nondecreasing. 

Since $z$ is proper, we know that $z(a)\neq -\infty$ for all $a$.  Thus given some $a_0>0$, there must exist some $b_0\in \RR$ such that $b_0\leq \frac{z(a_0)}{a_0}$. It then follows that for all $a\geq a_0$
\[
ab_0-z(a)\leq ab_0-z(a_0)-(a-a_0)\frac{z(a_0)}{a_0}=a(b_0-\frac{z(a_0)}{a_0})\leq 0.
\]
Therefore, for all $b\leq b_0$
\[
\sup_{a\in\RR} ab-z(a)=\sup_{a\in [0,a_0]} ab-z(a).
\]

Fix $\epsilon>0$ and let $a_n\in [0,a_0]$ be a decreasing sequence such that $z^*(-n)\leq \epsilon-n a_n-z(a_n)$ (note that from the above logic such choices of $a_n$ must exist once $n$ is sufficiently large).  Since $a_n$ is decreasing and bounded from below, it must converge to a limit point $\bar{a}$ as $n\to\infty$. Thus, 
\[
0\leq \liminf_{n\to \infty} z^*(-n)\leq \epsilon-z(\bar{a})-\limsup_{n\to\infty} na_n,
\]
which immediately implies that $\bar{a}=0$. We can then rewrite the above as
\[
 \liminf_{n\to \infty} z^*(-n)\leq \epsilon-\limsup_{n\to\infty} na_n\leq \epsilon.
 \]
 Therefore, $\liminf_{n\to\infty} z^*(-n)=0.$
 
It now follows that if $z^*(b)\in (0,\infty)$, then there must exist some $b_0<b$ such that $2z^*(b_0)\leq z^*(b)$. We then have
\[
\inf \partial z^*(b)\geq \frac{z^*(b)}{2(b-b_0)}>0.
\]
Thus, $z^*$ is strictly increasing at $b$ whenever $z^*(b)\in (0,\infty)$.  Hence $(z^*)^{-1}$ is well defined and Lipschitz on $z^*(\RR)\cap (0,\infty)$. 
\end{proof}

While the invertibility of $q=z^*(p)$ can fail when $z^*(p)=0$, this will not cause a problem for our study of the systems (\ref{eq:system}) and (\ref{eq:system_q}), as the failure cannot happen on the support of $\rho$.

\begin{lemma} Suppose that $z$ satisfies assumptions (z1-z3).
If $(z^*)^{-1}$ cannot be extended to a continuous function on $[0,\infty)\cap z^*(\RR)$, then $\partial z^*(p)=\{0\}$ whenever $z^*(p)=0$.
\end{lemma}
\begin{proof}
Let $p_0=\sup\{p\in \RR: z^*(p)=0\}$. If $p_0=-\infty$, then the statement is vacuously true.

Otherwise, we define $(z^*)^{-1}(0)=p_0$. If $z^*(\RR)\cap [0,\infty)=\{0\}$, then $(z^*)^{-1}$ is trivially continuous on $[0,\infty)\cap z^*(\RR)$.  Thus, we only need to worry about the case where $z^*(\RR)\cap (0,\infty)\neq \varnothing$ and
 there exists $a_0\in \partial z^*(p_0)$ such that $a_0>0$.   Convexity then implies that for any $p>p_0$ with $z^*(p)\neq +\infty$ we have $\inf\partial z^*(p)\geq a_0$. Thus, the Lipschitz constant of $(z^*)^{-1}$ must be bounded in a neighborhood of zero and therefore the extension $(z^*)^{-1}(0)=p_0$ must be continuous.
\end{proof}

Perhaps the most significant aspect of the change of variables $q=z^*(p)$ is the change in the energy controlling the primal and dual coupling.  Recall that we defined the new energy $e$ through the formula
\begin{equation}\label{eq:e_def}
e(a)=\begin{cases}
az(a)-2\int_0^a z(s)\, ds & \textup{if}\;\; z(a)\neq +\infty,\\
+\infty & \textup{otherwise.}
\end{cases}
\end{equation} 
While this formula appears somewhat mysterious, $e$ is the unique (up to an irrelevant constant factor) convex function such that  $\partial e(a)=z^*\circ \partial z(a)$ when $\partial z(a)\neq \varnothing$. Thus, when $p\in \partial z(\rho)$ we will know that $q\in \partial e(\rho)$.    Note that the monotonicity of $z^*$ is key, otherwise $e$ would fail to be convex. The following Lemma records the properties that $e$ inherits from $z$.

\begin{lemma}
Suppose that $z$ is an energy satisfying (z1-z3). If we define $e:\RR\to\RR\cup\{+\infty\}$ according to (\ref{eq:e_def}),
then $e$ satisfies the following properties
\begin{enumerate}[(e1)]
\item  $e:\RR\to\RR\cup\{+\infty\}$ is proper, convex, and lower semicontinuous.
\item $e(a)=+\infty$ if $a<0$, $e(0)=0$,  and $e$ is increasing on $e^{-1}(\RR)$.
\item $\limsup_{a\to 0^+} \frac{e(a)}{a}=0,$ and $\liminf_{b\to\infty} \frac{e^*(b)}{b}>0$.  
\end{enumerate}
Furthermore, if $a\neq 0$, we have
\[
\partial e(a)=\{ab-z(a):b\in \partial z(a)\}=\{z^*(b):b\in \partial z(a)\},
\]
and so $\partial e(a)$ is a singleton if and only if $\partial z(a)$ is a singleton. 
\end{lemma}
\begin{proof}
It is clear that $e(0)=0$ and $e(a)=+\infty$ if $z(a)=+\infty$.

Given any two points $a_0, a_1\in z^{-1}(\RR)$, convexity implies that
\begin{equation}\label{eq:integral_convexity}
2(a_1-a_0)z(\frac{a_1+a_0}{2})\leq 2\int_{a_0}^{a_1} z(s)\, ds\leq (a_1-a_0)(z(a_0)+z(a_1)).
\end{equation}
Thus, if $z(a)\neq +\infty$, then
\[
0 \leq e(a)\leq az(a)-2az(\frac{a}{2})<\infty.
\]
Therefore $e(a)=+\infty$ if and only if $z(a)=+\infty$.  Thus, the set $e^{-1}(\RR)$ is an interval. Furthermore, the above inequalities combined with $(z3)$ clearly imply that $\limsup_{a\to 0^+}\frac{e(a)}{a}=0.$

Again using (\ref{eq:integral_convexity}),
\[
e(a_1)-e(a_0)=a_0(z(a_1)-z(a_0))+(a_1-a_0)z(a_1)-2\int_{a_0}^{a_1} z(s)\, ds\geq a_0(z(a_1)-z(a_0))-(a_1-a_0)z(a_0)
\]
If $b_0\in \partial z(a_0)$, then
\[
e(a_1)-e(a_0)\geq (a_1-a_0)\big(a_0b_0-z(a_0)).
\]
Thus, $b\in \partial z(a)$ implies that $ab-z(a)\in \partial e(a)$ whenever $a\in e^{-1}(\RR)$.  Thus, the subdifferential of $e$ is nonempty whenever the subdifferential of $z$ is nonempty.  Combining this with the equality $z^{-1}(\RR)=e^{-1}(\RR)$, it follows that $e$ is convex, lower semicontinuous and proper. 

Note that $b\in \partial z(a)$ implies that $z^*(b)=ab-z(a)$.  Therefore, $\{ab-z(a): a\in \partial z(a)\}=\{z^*(b):b\in \partial z(a)\}$.  Since $\int_0^a z(s)\, ds$ is everywhere differentiable on the interior of $z^{-1}(\RR)$, every element of $\partial e(a)$ must have the form $ab-z(a)$ for $b\in \partial z(a)$. Convexity implies that $ab-z(a)\geq -z(0)=0$, thus $e$ is increasing on the interior $e^{-1}(\RR)$.

It remains to show that $\lim_{b\to \infty} \frac{e^*(b)}{b}>0$.  Since $\limsup_{a\to 0^+}\frac{e(a)}{a}=0$, there must exist some $a_0>0$ such that $e(a_0)<\infty$. Thus,
\[
\liminf_{b\to\infty}\frac{e^*(b)}{b}\geq \liminf_{b\to\infty} a_0-\frac{e(a_0)}{b}=a_0.
\]

\end{proof}

\begin{table}[h]
    \centering
    \begin{tabular}{|c|c|c|c|c|}
    \hline
   Parameter & $z$ energy $a\in [0,\infty)$ & $z^*$ energy $b\in \RR$  & $e$ energy $a\in [0,\infty)$  & $e^*$ energy $b\in \RR$ \\
      \hline
         $m\in (0,\infty]\setminus \{1\}$ & $\frac{1}{m-1}(a^m-a) $        & $\max(\frac{(m-1)b+1}{m},0)^{m/(m-1)}$ &$\frac{1}{m+1}a^{m+1}$ & $\frac{m}{m+1}\max(b,0)^{\frac{m+1}{m}}$\\
          \hline
          $m\to 1$ &  $a\log(a)-a $  & $\exp(b)$       & $\frac{1}{2}a^2$  & $\frac{1}{2}\max(b,0)^2$\\
         \hline
    \end{tabular}
    \caption{Some examples of the transformation from $z$ to $e$.}
    \label{tab:transformation}
\end{table}

Now that we have established properties of the transformation $q=z^*(p)$ we can temporarily forget about the original system (\ref{eq:system}) and focus on (\ref{eq:system_q}). We will eventually return to (\ref{eq:system}) in the final section, where we show that solutions to (\ref{eq:system_q}) can be transformed into solutions to (\ref{eq:system}).  Until then, our efforts will be concentrated on establishing the energy dissipation strategy described in the introduction.

\section{Convex analysis and compensated compactness}\label{sec:cc}

In this section, we collect some results that we will need to establish the weak convergence of the primal and dual energy terms.  
We begin by defining some convex spaces that we will work with throughout the paper.
\begin{definition}
Given an energy $e$ satisfying (e1-e3), we define
\[
X(e):=\{\rho\in L^{1}_{\loc}(Q_{\infty}): e(\rho)\in L^{\infty}_{\loc}([0,\infty);L^1(\RR^d))\},
\]
\[
Y(e^*):=\{q\in L^{1}_{\loc}(Q_{\infty}): e^*(q)\in L^1_{\loc}([0,\infty);L^1(\RR^d))\}.
\]
\end{definition}
We are now ready to introduce a result that is one of the cornerstones of our argument.
\begin{prop}\label{prop:e_e^*}
 Let $e:\RR\to\RR\cup\{+\infty\}$ be an energy satisfying $(e1-e3)$.
Let $e_k:\RR\to\RR\cup\{+\infty\}$ be a sequence of energies satisfying (e1-e3)  such that $e_k$ converges pointwise everywhere to $e$. Suppose we have a sequence of nonnegative density and pressure functions $\rho_k\in X(e_k)$,  $q_k\in Y(e^*_k)$ such that  $\rho_kq_k=e_k(\rho_k)+e^*_k(q_k)$ almost everywhere and $\rho_k, q_k$ converge weakly in $L^1_{\loc}(Q_{\infty})$ to  limits $\rho, q\in L^1_{\loc}(Q_{\infty})$ respectively. 
If $\rho q\in L^1_{\loc}([0,\infty);L^1(\RR^d))$ and  for every nonnegative $\vp\in C^{\infty}_c(Q_{\infty})$
\[
\limsup_{k\to\infty} \int_{Q_{\infty}}\vp\rho_kq_k\leq \int_{Q_{\infty}} \vp\rho q, 
\]
then $\rho\in X(e), q\in Y(e^*)$,  $\rho q =e(\rho)+e^*(q)$ almost everywhere, and
$\rho_kq_k, e_k(\rho_k), e_k^*(q_k)$ converge weakly in $L^1_{\loc}([0,\infty);L^1(\RR^d))$ to $\rho q, e(\rho), e^*(q)$ respectively.  
\end{prop}

\begin{proof}
Given some nonnegative $\vp\in C_c^{\infty}(Q_{\infty})$, let $D$ be a compact set containing the support of $\vp$.
From our assumptions, we have
\[
\int_{Q_{\infty}} \vp \rho q\geq \limsup_{k\to\infty} \int_{Q_{\infty}} \vp\rho_kq_k=\limsup_{k\to\infty} \int_{Q_{\infty}} \vp e_k(\rho_k)+\vp e^*_k(q_k).
\]
Fix some simple functions $g_1, g_2\in L^{\infty}(D)$ such that every value of $g_1$ is a value where $e_k^*$ converges to $e^*$ (c.f. Lemma \ref{lem:convex_convergence}).  It then follows that
\[
\limsup_{k\to\infty} \int_{Q_{\infty}} \vp\big( e_k(\rho_k)+ e^*_k(q_k)\big)\geq \limsup_{k\to\infty} \int_{Q_{\infty}} \vp\big(g_1\rho_k-e^*_k(g_1)+g_2q_k-e_k(g_2)   \big)= \int_{Q_{\infty}} \vp\big(g_1\rho-e^*(g_1)+g_2q-e(g_2)   \big).
\]
Taking a supremum over $g_1, g_2$, we can conclude that
\[
\int_{Q_{\infty}} \vp \rho q\geq \limsup_{k\to\infty} \int_{Q_{\infty}} \vp\big( e_k(\rho_k)+ e^*_k(q_k)\big)\geq\int_{Q_{\infty}} \vp\big(e(\rho)+e^*(q)\big).
\]
On the other hand, Young's inequality immediately implies that 
\[
\rho q\leq e(\rho)+e^*(q)
\]
almost everywhere. 
Thus,  $\rho q=e(\rho)+e^*(q)$ almost everywhere.
This also now implies that $\rho\in X(e)$ and $q\in Y(e^*)$.

The previous calculation shows that $e_k(\rho_k)+e_k^*(q_k)$ is uniformly bounded in $L^1_{\loc}([0,\infty);L^1(\RR^d))$. Thus, for any time $T>0$, there exists $w_1, w_2\in C(Q_T)^*$ such that $e_k(\rho_k), e_k^*(q_k)$ converge (along a subsequence that we will not relabel) to $w_1, w_2$ respectively.
Arguing as in the first paragraph, it follows that
\[
 \int_{Q_T} \vp w_1 =\liminf_{k\to\infty} \int_{Q_T} \vp e_k(\rho_k)\geq \int_{Q_T} \vp e(\rho), \quad  \int_{Q_T} \vp w_2= \liminf_{k\to\infty} \int_{Q_T} \vp e_k^*(q_k)\geq \int_{Q_T} \vp e^*(q).
\]
 Hence,
\[
\int_{Q_T}   \vp |w_1-e(\rho)|+\vp |w_2-e^*(q)|=\int_{Q_T} \vp \big(w_1-e(\rho)+w_2-e^*(q)\big)=
\]
\[
 \limsup_{k\to\infty}  \int_{Q_T} \vp \big(e_k(\rho_k)+e^*_k(q_k)-e(\rho)-e^*(q)\big)= \limsup_{k\to\infty}  \int_{Q_T} \vp \big(\rho_k q_k -\rho q\big)\leq 0.
\]
Thus, $w_1=e(\rho)$ and $w_2=e^*(q)$.
Since $w_1, w_2$ and $T>0$ were arbitrary, it follows that $e(\rho), e^*(q)$ are the only weak limit points of $e_k(\rho_k), e_k^*(q_k)$ in $L^1_{\loc}([0,\infty);L^1(\RR^d))$.  Thus, the full sequences $e_k(\rho_k), e_k^*(q_k)$ must converge weakly in $L^1_{\loc}([0,\infty);L^1(\RR^d))$ to $e(\rho)$ and $e^*(q)$ respectively.  The weak $L^1_{\loc}([0,\infty);L^1(\RR^d))$ convergence of $\rho_kq_k$ to $\rho q$ is an immediate consequence.

\end{proof}


Of course, to even be able to use Proposition \ref{prop:e_e^*}, we somehow need to know an upper semicontinuity type property for the product $\rho_kq_k$.  In practice, this seems to require establishing the weak convergence of $\rho_kq_k$ to $\rho q$.   Luckily, the following ``compensated compactness''-type Lemma shows that the weak convergence of the product can hold even when the strong convergence of both $\rho_k$ and $q_k$ is unknown.  Unlike typical compensated compactness arguments that decompose the codomain of the function, the following compensated compactness argument is based on a decomposition of the domain of the functions.  Indeed, we show that if $\rho_k$ has some time regularity and $q_k$ has some space regularity then their product weakly converges.  This argument was inspired by the proof of the main Theorem in \cite{santambrogio_crowd_motion}, although we would not be surprised if this result was already established in an earlier work. 

\begin{lemma}\label{lem:spacetime_cc}
Fix some $r\in (1,\infty)$ and let $r'$ be the Holder conjugate of $r$. Let $Z_r=L^r_{\loc}(Q_{\infty})\times L^{r'}_{\loc}(Q_{\infty})$ and let $\eta$ be a spatial mollifier.
Suppose that $(u_k,v_k)\in Z_r$ is a sequence that converges weakly in $Z_r$ to a limit point $(u,v)\in Z_r$.  If $u_k$ is equicontinuous with respect to space in $L^r_{\loc}(Q_{\infty})$ and for any $\epsilon>0$,  $\eta_{\epsilon}*v_k$ is equicontinuous with respect to space and time in $L^{r'}_{\loc}(Q_{\infty})$, then $u_kv_k$ converges weakly in $(C_c(Q_{\infty}))^*$ to $uv$.
\end{lemma}
\begin{proof}
  Define $v_{k,\epsilon}:=\eta_{\epsilon}*v_k$ and $v_{\epsilon}:=\eta_{\e}*v$.  For $\epsilon>0$ fixed and any compact set $D\subset Q_{\infty}$, the Riesz-Frechet-Kolmogorov compactness theorem  implies that $v_{k,\epsilon}$ converges strongly in $L^{r'}(D)$ to $v_{\epsilon}$ as $k\to\infty$.

Given $\vp\in C_c^{\infty}(Q_{\infty})$, we must have
\[
\lim_{\epsilon\to 0} \int_{Q_{\infty}} \vp (v-v_{\epsilon})u=0,
\]
and 
\[
\lim_{k\to\infty } \int_{Q_{\infty}} \vp (v_{k,\epsilon}-v_{\epsilon})u_k+v_{\epsilon}(u-u_k)=0.
\]
Thus, to prove the weak convergence of $u_kv_k$ to $uv$, it will suffice to show that
\[
\lim_{\epsilon\to 0} \lim_{k\to\infty} \int_{Q_{\infty}} \vp(v_k-v_{k,\epsilon})u_k=0.
\]
Rearranging the convolution, this is equivalent to showing 
\[
\lim_{\epsilon\to 0} \lim_{k\to\infty} \int_{Q_{\infty}} v_k \big(\eta_{\epsilon}*\vp u_k-\vp u_k\big)=0.
\]

Choose some compact set $D\subset Q_{\infty}$ such that for any $\epsilon$ sufficiently small, the support of $\vp, \eta_{\epsilon}*\vp$ is contained in $D$.  We then have the estimate
\[
\Big| \int_{Q_{\infty}} v_k \big(\eta_{\epsilon}*\vp u_k-\vp u_k\big)\Big|\lesssim \norm{v_k}_{L^{r'}(D)}\big(\norm{\vp}_{L^{\infty}(Q_{\infty})}\norm{u_k-\eta_{\epsilon}*u_k}_{L^r(D)}+\epsilon\norm{u_k}_{L^r(D)}\norm{\nabla \vp}_{L^{\infty}(Q_{\infty})}\big).
\]
The weak convergence of $(u_k,v_k)$ to $(u,v)$ in $Z_r$ implies that $\norm{u_k}_{L^r(D)}+\norm{v_k}_{L^{r'}(D)}$ is bounded with respect to $k$.  Spatial equicontinuity gives us
\[
\lim_{\epsilon\to 0} \sup_k\norm{u_k-\eta_{\epsilon}*u_k}_{L^r(D)}=0.
\]
Thus, it follows that 
\[
\lim_{\epsilon\to 0}\sup_k\Big| \int_{Q_{\infty}} v_k \big(\eta_{\epsilon}*\vp u_k-\vp u_k\big)\Big|=0,
\]
and so we can conclude that $u_kv_k$ converges in $(C_c(Q_{\infty}))^*$ to $uv$.
\end{proof}

\section{Energy dissipation and estimates}\label{sec:estimates}

We will now begin to analyze the parabolic structure of the equation (\ref{eq:parabolic_q}).  In order to do this, we will need to upgrade the spaces $X(e), Y(e^*)$ into spaces that are more appropriate for solving PDEs

\begin{definition}
Given an energy $e$ satisfying (e1-e3), we define
\[
\cX(e):=\{\rho\in X(e): \rho\in L^{\infty}_{\loc}([0,\infty);L^1(\RR^d)\cap L^{\infty}(\RR^d))\cap H^{1}_{\loc}([0,\infty);H^{-1}(\RR^d))\},
\]
\[
\cY(e^*):=\{q\in Y(e^*): q\in L^{\frac{2d+4}{d+4}}_{\loc}([0,\infty);L^2_{\loc}(\RR^d)) \cap L^2_{\loc}([0,\infty);\dot{H}^{1}(\RR^d))\}.
\]
\end{definition}
Note that the seemingly strange choice of time integrability for $\cY$ will become clear later.

\begin{prop}\label{prop:edr}
Given an energy $e:\RR\to\RR\cup\{+\infty\}$ satisfying (e1-e3), suppose that $e(\rho^0)\in L^1(\RR^d)\cap L^{\infty}(\RR^d)$ and $\rho^0\in L^1(\RR^d)$. 
Let $\rho\in \cX(e)$ be a density function  and $q\in \cY(e^*)$ a pressure function that satisfy the duality relation $\rho q=e(\rho)+e^*(q)$ almost everywhere.  Suppose that $\mu\in L^{\infty}(\frac{1}{\rho})$ is a growth rate and $V\in L^2_{\loc}([0,\infty);L^2(\RR^d))$ is a vector field such that $\nabla \cdot V\in L^{\infty}(Q_{\infty})$.  If for every $\psi\in W^{1,1}_c([0,\infty);\cY(e^*))$, $\rho, q$ are weak solutions of the parabolic equation
\begin{equation}\label{eq:weak_parabolic}
\int_{\RR^d} \psi(0,x) \rho^0(x)\, dx= \int_{Q_{\infty}} \nabla q\cdot \nabla \psi-\rho\partial_t\psi-\rho V\cdot\nabla \psi-\mu \psi,
\end{equation}
 then for any nonnegative $\omega\in W^{1,\infty}_c([0,\infty))$ that depends only on time, we have the dissipation relation
\begin{equation}\label{eq:limiting_upper_inequality}
\int_{\RR^d}\omega(0) e(\rho^0(x))\, dx= \int_{Q_{\infty}} -e(\rho)\partial_t \omega+\omega|\nabla q|^2+\omega e^*(q)\nabla \cdot V-\omega\mu q .
\end{equation}

\end{prop}
\begin{proof}
Let $\tilq\in C^{\infty}_c(\RR^d)$  such that $e^*(\tilq)\in L^1(\RR^d)$.  Extend $q$ backwards in time by defining $q(-t,x)=\tilq(x)$ for all $t\in (0,\infty)$. Fix $\epsilon>0$, and define
\[
q_{\epsilon}(t,x):=\frac{1}{\epsilon}\int_{t-\epsilon}^{t} q(s,x)\, ds
\]
for all $(t,x)\in \RR\times\RR^d$. 

By Jensen's inequality, $q_{\epsilon}\in \cY(e^*)$ and a direct computation shows that $\partial_t q_{\epsilon}$ is the linear combination of two $\cY(e^*)$ functions for any $\epsilon>0$.
Given any nonnegative $\omega\in W^{1,\infty}_c([0,\infty))$ that is a function of time only, it now follows that $q_{\epsilon }\omega$ is a valid test function for the weak equation  (\ref{eq:weak_parabolic}).   Thus, we have 
\begin{equation}\label{eq:weak_with_p_test}
\int_{\RR^d} q_{\epsilon}(0,x)\omega(0) \rho^0(x)\, dx= \int_{Q_{\infty}} -\rho\partial_t(\omega q_{\epsilon})+(\nabla q-\rho V)\cdot\nabla( q_{\epsilon}\omega)-\mu \omega q_{\epsilon},
\end{equation}

Note that for almost every $(t,x)\in Q_{\infty}$
\[
\rho\partial_t (\omega q_{ \epsilon})=\rho(t,x) q_{ \epsilon}(t,x) \partial_t \omega(t,x) +\omega(t,x) \frac{q(t,x)-q(t-\epsilon,x)}{\epsilon}\rho(t,x).
\]
Hence, we can apply Young's inequality to deduce that
\begin{equation}\label{eq:key_energy_inequality}
(\frac{q(t,x)-q(t-\epsilon,x)}{\epsilon})\rho(t,x)\geq \frac{e^*(q(t,x))-e^*(q(t-\epsilon,x))}{\epsilon}
\end{equation}
By defining
\[
(e^*(q))_{\epsilon}:=\frac{1}{\epsilon}\int_{t-\epsilon}^t e^*(q(s,x))\, ds
\]
we can write the above inequality in the more compact form 
\[
\rho\partial_t q_{\epsilon} \geq \partial_t (e^*(q))_{\epsilon}
\]

Plugging this into (\ref{eq:weak_with_p_test}), we get the inequality
\[
\int_{\RR^d} q_{\epsilon}(0,x)\omega(0) \rho^0(x)\, dx\leq   \int_{Q_{\infty}} -\rho q_{\epsilon}\partial_t \omega-\omega \partial_t (e^*(q))_{\epsilon}+(\nabla q-\rho V)\cdot\nabla( q_{\epsilon}\omega)-\mu \omega q_{\epsilon}, 
\]
Moving time derivatives back on to $\omega$, we get the equivalent inequality
\begin{equation}\label{eq:reusable}
\int_{\RR^d}\omega(0)  \Big(q_{\epsilon}(0,x)\rho^0(x)-e^*\big(q_{\epsilon}(0,x)\big)\Big)\, dx
\end{equation}
\[
\leq  \int_{Q_{\infty}}  \partial_t \omega( (e^*(q))_{\epsilon}-\rho q_{\epsilon})+(\nabla q-\rho V)\cdot\nabla( q_{\epsilon}\omega)-\mu \omega q_{\epsilon}. 
\]
Note that we also have 
\[
\int_{\RR^d}\omega(0)  \Big(q_{\epsilon}(0,x)\rho^0(x)-e^*\big(q_{\epsilon}(0,x)\big)\Big)\, dx=\int_{\RR^d}\omega(0)  \Big(\tilq(x)\rho^0(x)-e^*\big(\tilq(x)\big)\Big)\, dx
\]
thanks to our construction of $q_{\epsilon}$.

Since all of the time derivatives are now on $\omega$, we can safely send $\epsilon\to 0$. Thus, it follows that
\[
\int_{\RR^d}\omega(0)  \Big(\tilq(x)\rho^0(x)-e^*\big(\tilq(x)\big)\Big)\, dx
\]
\[
\leq  \int_{Q_{\infty}}  \partial_t \omega( e^*(q)-\rho q)+\omega|\nabla q|^2+\omega e^*(q)\nabla \cdot V-\mu \omega q 
\]
where we have used the fact that $\nabla e^*(q)=\rho\nabla q$ (note that this is just a consequence of the chain rule for Sobolev functions).
Exploiting the duality relation $\rho q=e(\rho)+e^*(q)$, we have arrived at the inequality
\begin{equation}\label{eq:almost_upper_inequality}
\int_{\RR^d}\omega(0)  \Big(\tilq(x)\rho^0(x)-e^*\big(\tilq(x)\big)\Big)\leq  \int_{Q_{\infty}} -e(\rho)\partial_t \omega+\omega|\nabla q|^2+\omega e^*(q)\nabla \cdot V-\omega\mu q .
\end{equation}
$\tilq$ was arbitrary, thus, taking a supremum over $\tilq$ we obtain one direction of the dissipation relation.

To get the other direction, we instead smooth $q$ forwards in time by defining
\[
\bar{q}_{\epsilon}:=\frac{1}{\epsilon}\int_t^{t+\epsilon} q(s,x).
\]
The argument will then proceed identically to the above except that the forward-in-time smoothing does not allow us to conclude that $q_{\epsilon}(0,x)=\tilde{q}$. Luckily, Young's inequality is now in our favor and so we just use
\[
\int_{\RR^d}\omega(0)  \Big(q_{\epsilon}(0,x)\rho^0(x)-e^*\big(q_{\epsilon}(0,x)\big)\Big)\, dx\leq \int_{\RR^d}\omega(0)e(\rho^0(x)) dx.
\]
\end{proof}

In the next proposition we collect some a priori estimates for solutions to (\ref{eq:parabolic_q}).
In fact, we will consider a slightly more general equation where we add an additional viscosity term $-\gamma \Delta \rho$ where the constant $\gamma$ is possibly zero.  As we will see, the estimates will give us uniform control when we consider sequences of solutions.

\begin{prop}\label{prop:estimates}
Let $e$ be an energy function satisfying (e1-e3), let $V\in L^2_{\loc}(Q_{\infty})$ be a vector field such that $\nabla \cdot V\in L^{\infty}(Q_{\infty})$, let $\frac{\mu}{\rho}\in L^{\infty}(Q_{\infty})$ and let $\gamma$ be a positive constant.   Suppose that $\rho\in \cX(e)\cap L^2_{\loc}([0,\infty);H^1(\RR^d))$, $q\in \cY(e^*)$ satisfy the duality relation $\rho q=e(\rho)+e^*(q)$ almost everywhere.
If $e(\rho^0)\in L^1(\RR^d)$ and the variables satisfy the weak equation
\begin{equation}\label{eq:viscous_parabolic}
\int_{\RR^d} \psi(0,x) \rho^0(x)\, dx= \int_{Q_{\infty}} \gamma\nabla \rho\cdot \nabla \psi+\nabla q\cdot \nabla \psi-\rho\partial_t\psi-\rho V\cdot\nabla \psi-\mu \psi,
\end{equation}
for every test function $\psi\in W^{1,1}_c([0,\infty);L^1(\rho)\cap \dot{H}^1(\RR^d))$, then 
 for any nonnegative $\omega\in W^{1,\infty}_c([0,\infty))$ that depends only on time and for every $m\in (1,\infty)$, we have the dissipation inequalities
\begin{equation}\label{eq:viscous_edi}
     \int_{Q_{\infty}} -e(\rho)\partial_t \omega+\omega|\nabla q|^2+\omega e^*(q)\nabla \cdot V-\omega\mu q \leq \int_{\RR^d}\omega(0) e(\rho^0(x))\, dx
\end{equation}
\begin{equation}\label{eq:viscous_edi_rho}
   \int_{Q_{\infty}}  \omega\gamma (m-1)\rho^{m-2}|\nabla \rho|^2- \rho^{m}\big( \frac{1}{m}\partial_{t}\omega  +\omega(\frac{\mu}{\rho}-\frac{m-1}{m}\nabla \cdot V )\big) \leq \int_{\RR^d} \frac{\omega(0)}{m}(\rho^0)^{m}\, dx
\end{equation}
and if we set $\beta=\inf\{b\in \RR: e^*(b)\geq 1\}$
then
the following estimates hold for almost all $T\in [0,\infty)$:
\begin{equation}\label{eq:gamma_nabla_rho}
\gamma\norm{\nabla \rho}_{L^2(Q_T)}^2\leq \norm{\rho^0}_{L^2(\RR^d)}^2+\norm{\rho}_{L^2(Q_T)}^2\big(\norm{\frac{\mu}{\rho}}_{L^{\infty}(Q_T)}+\norm{\nabla \cdot V}_{L^{\infty}(Q_{\infty})}),
\end{equation}
\begin{equation}\label{eq:l1_growth}
\norm{\rho(T\,\cdot)}_{L^1(\RR^d)}\leq \norm{\rho^0}_{L^1(\RR^d)}\exp(T\norm{\frac{\mu}{\rho}}_{L^{\infty}(Q_T)})
\end{equation}
\begin{equation}\label{eq:rho_h_minus1}
\norm{\partial_t \rho}_{L^2([0,T];H^{-1}(\RR^d))}\leq \gamma\norm{\nabla \rho}_{L^2(Q_T)}+\norm{\nabla q}_{L^{2}(Q_T)}+\norm{\mu}_{L^2(Q_T)}+\norm{\rho V}_{L^2(Q_T)}
\end{equation}
\begin{equation}\label{eq:rho_linf}
\norm{\rho(T,\cdot )}_{L^{\infty}(\RR^d)}\leq \norm{\rho^0}_{L^{\infty}(\RR^d)}\exp\big( 2T(\norm{\nabla \cdot V}_{L^{\infty}(Q_T)}+\norm{\frac{\mu}{\rho}}_{L^{\infty}(Q_T)}) \big),
\end{equation}
\begin{equation}\label{eq:nabla_q_control}
 \norm{\nabla q}_{L^2(Q_T)}^2 
\lesssim_d \int_{\RR^d}e(\rho^0)\, dx+\max(\beta,1)\Big( \norm{\rho}_{L^1(Q_T)}+ \norm{\rho}_{L^{\infty}[0,T];L^1(\RR^d)}^{\frac{2}{d}}\norm{\rho}_{L^2(Q_T)}^2\Big)\big(1+\norm{\frac{\mu}{\rho}}_{L^{\infty}(Q_T)}+\norm{\nabla \cdot V}_{L^{\infty}(Q_T)}\big)^2,
\end{equation}
\begin{equation}\label{eq:rho_p_control}
\norm{e^*(q)}_{L^1(Q_T)} +\norm{e(\rho)}_{L^1(Q_T)}\lesssim_d \beta \norm{\rho}_{L^1(Q_T)}+(\beta \norm{\rho}_{L^{\infty}[0,T];L^1(\RR^d)})^{\frac{1}{d}}\big(\norm{\rho}_{L^2(Q_T)}\norm{\nabla q}_{L^2(Q_T)}\big),
\end{equation}
\begin{equation}\label{eq:dual_energy_extra_control}
  \norm{e^*(q)}_{L^{\frac{2d+4}{d+4}}([0,T];L^2(\RR^d))}\lesssim_d \norm{e^*(q)}_{L^1(Q_T)}^{\frac{2}{d+2}}\norm{\nabla q}_{L^{2}(Q_T)}^{\frac{d}{d+2}}\norm{\rho}_{L^{\infty}(Q_T)}^{\frac{d}{d+2}}.
\end{equation}
and for any compact set $K\subset \RR^d$,
\begin{equation}\label{eq:rho_p_extra_control}
  \norm{q}_{L^{\frac{2d+4}{d+4}}([0,T];L^2(K))}\lesssim_d \beta T |K|+ \beta\norm{e^*(q)}_{L^1(Q_T)}^{\frac{2}{d+2}}\norm{\nabla q}_{L^{2}(Q_T)}^{\frac{d}{d+2}}\norm{\rho}_{L^{\infty}(Q_T)}^{\frac{d}{d+2}}.
\end{equation}
\end{prop}
\begin{proof}

The dissipation inequalities (\ref{eq:viscous_edi}) and (\ref{eq:viscous_edi_rho}) follow from choosing the test functions $q$ and $\rho^{m-1}$ respectively.  These test functions do not have the required time regularity, however, by following an identical argument to Proposition \ref{prop:edr}, this technicality can be overcome. In addition, note that in both inequalities we have dropped a term involving $ \nabla \rho\cdot \nabla q$, which is nonnegative thanks to the duality relation.

Estimates (\ref{eq:l1_growth}) and (\ref{eq:rho_h_minus1}) are straightforward consequences of the weak equation (\ref{eq:viscous_parabolic}). Estimate (\ref{eq:gamma_nabla_rho}) follows from (\ref{eq:viscous_edi_rho}) with $m=2$.   Estimate (\ref{eq:rho_linf}) follows from applying a Gronwall argument to (\ref{eq:viscous_edi_rho}) and then sending $m\to\infty$.

The estimates (\ref{eq:nabla_q_control}-\ref{eq:rho_p_extra_control}), are all linked. We begin by fixing a time $T\in [0,\infty)$ and considering $\norm{\rho q}_{L^1(Q_T)}$. Define $\tilq:=\max(q,\beta)-\beta$. It is then clear that
\[
\norm{\rho q}_{L^1(Q_T)}\leq \beta\norm{\rho}_{L^1(Q_T)}+\norm{\rho\tilq}_{L^1(Q_T)}, \quad \norm{\nabla \tilq}_{L^2(Q_T)}\leq \norm{\nabla q}_{L^2(Q_T)}.
\]
Working in Fourier space, we have 
\[
\norm{ \rho \tilq}_{L^1(Q_T)}\leq \int_0^T \int_{\RR^d} |\hat{\rho}(t,\xi)\hat{\tilq}(t,\xi)|\, d\xi \, dt\leq \int_0^T|B_R|\norm{\rho(t,\cdot)}_{L^1(\RR^d)}\norm{\tilq(t,\cdot)}_{L^1(\RR^d)}+\int_{|\xi|>R} |\hat{\rho}(t,\xi)\hat{\tilq}(t,\xi)|\, d\xi \, dt
\]
\[
\leq T|B_R|\norm{\rho}_{L^{\infty}([0,T];L^1(\RR^d))}\norm{\tilq}_{L^1(Q_T)}+R^{-1}\norm{\rho}_{L^2(Q_T)}\norm{\nabla \tilq}_{L^2(Q_T)}
\]
where $R>0$ and $B_R$ is the ball of radius $R$. Optimizing over $R$ and dropping dimensional constants, it follows that
\[
\int_{Q_T} \rho \tilq\lesssim_d \big(\norm{\rho}_{L^{\infty}([0,T];L^1(\RR^d))}\norm{\tilq}_{L^1(Q_T)}\big)^{\frac{1}{d+1}}\big(\norm{\rho}_{L^2(Q_T)}\norm{\nabla q}_{L^2(Q_T)}\big)^{\frac{d}{d+1}}.
\]
If $b>\beta$ and $ e^*(b)\neq +\infty$, then it follows from the definition of $\beta$ that $\beta^{-1}= \liminf_{\epsilon\to 0^+} \frac{e^*(\beta+\epsilon)-e^*(0)}{\beta+\epsilon}\leq \inf \partial e^*(b)$.  Therefore,

\[
\max(e^*(q)-e^*(\beta),0)\geq \beta^{-1}\tilq
\]
It then follows that 
\[
\norm{\tilq}_{L^1(Q_T)}\leq \beta\norm{e^*(q)}_{L^1(Q_T)}\leq \beta\norm{\rho q}_{L^1(Q_T)}.
\]
As a result,
\[
\norm{\rho q}_{L^1(Q_T)}\lesssim_d \beta\norm{\rho}_{L^1(Q_T)}+\big(\beta\norm{\rho}_{L^{\infty}([0,T];L^1(\RR^d))}\norm{\rho q}_{L^1(Q_T)}\big)^{\frac{1}{d+1}}\big(\norm{\rho}_{L^2(Q_T)}\norm{\nabla q}_{L^2(Q_T)}\big)^{\frac{d}{d+1}}
\]
Now using Young's inequality (suboptimally), it follows that
\[
\norm{\rho q}_{L^1(Q_T)}\lesssim_d \beta \norm{\rho}_{L^1(Q_T)}+(\beta \norm{\rho}_{L^{\infty}[0,T];L^1(\RR^d)})^{\frac{1}{d}}\big(\norm{\rho}_{L^2(Q_T)}\norm{\nabla q}_{L^2(Q_T)}\big).
\]
Since 
\[
\norm{e(\rho)}_{L^1(Q_T)}+\norm{e^*(q)}_{L^1(Q_T)}=\norm{\rho q}_{L^1(Q_T)}
\]
we have obtained the bound in (\ref{eq:rho_p_control}).

Now we turn to estimating $\norm{\nabla q}_{L^2(Q_T)}$.  From the dissipation relation (\ref{eq:viscous_edi}), we have 
\[
 \int_{Q_{\infty}}  \omega|\nabla q|^2- e(\rho)(\partial_{t}\omega+\frac{\mu}{\rho}\omega)  +\omega e^*(p)(\nabla \cdot V-\frac{\mu}{\rho}) \leq \int_{\RR^d} \omega(0)e(\rho^0)\, dx
\]
for any nonnegative $\omega\in W^{1,\infty}_c((0,\infty))$. Fix a time $T>0$ that is a Lebesgue point for the mapping $T\mapsto \norm{\nabla q}_{L^2(Q_T)}$.  Assume that $\omega$ is a decreasing function supported on $[0,T]$ and $\omega\leq 1$ everywhere. We can then eliminate the term $-e(\rho)\partial_t \omega$.   Thus, it follows from our previous work that
\[
 \int_{Q_{\infty}}  \omega|\nabla q|^2  \leq \int_{\RR^d}e(\rho^0)\, dx+\norm{\rho q}_{L^1(Q_T)}\big(\norm{\frac{\mu}{\rho}}_{L^{\infty}(Q_T)}+\norm{\nabla \cdot V}_{L^{\infty}(Q_T)}\big)
\]
\[
\lesssim_d \int_{\RR^d}e(\rho^0)\, dx+\max(\beta,1)\Big( \norm{\rho}_{L^1(Q_T)}+ \norm{\rho}_{L^{\infty}[0,T];L^1(\RR^d)}^{\frac{1}{d}}\big(\norm{\rho}_{L^2(Q_T)}\norm{\nabla q}_{L^2(Q_T)}\big)\Big)\big(\norm{\frac{\mu}{\rho}}_{L^{\infty}(Q_T)}+\norm{\nabla \cdot V}_{L^{\infty}(Q_T)}\big)
\]
If we let $\omega$ approach the characteristic function of $[0,T]$, then we deduce that $\norm{\nabla q}_{L^2(Q_T)}^2$ is 
 \[
 \lesssim_d \int_{\RR^d}e(\rho^0)\, dx+\Big(\beta \norm{\rho}_{L^1(Q_T)}+ (\beta\norm{\rho}_{L^{\infty}[0,T];L^1(\RR^d)})^{\frac{1}{d}}\norm{\rho}_{L^2(Q_T)}\norm{\nabla q}_{L^2(Q_T)}\Big)\big(\norm{\frac{\mu}{\rho}}_{L^{\infty}(Q_T)}+\norm{\nabla \cdot V}_{L^{\infty}(Q_T)}\big).
\]
Now we can use Young's inequality (suboptimally again) to get (\ref{eq:nabla_q_control})
\[
 \norm{\nabla q}_{L^2(Q_T)}^2 
\lesssim_d \int_{\RR^d}e(\rho^0)\, dx+\max(\beta,1)\Big( \norm{\rho}_{L^1(Q_T)}+ \norm{\rho}_{L^{\infty}[0,T];L^1(\RR^d)}^{\frac{2}{d}}\norm{\rho}_{L^2(Q_T)}^2\Big)\big(1+\norm{\frac{\mu}{\rho}}_{L^{\infty}(Q_T)}+\norm{\nabla \cdot V}_{L^{\infty}(Q_T)}\big)^2.
\]

 Finally, working in Fourier space again, it follows that for any  exponent $r\in [1,\frac{d+2}{2})$ and radius $R>0$,
 \[
 \norm{e^*(q)}_{L^r([0,T];L^2(\RR^d))}^r \lesssim_d \int_0^T\Big( R^d \norm{e^*(q(t,\cdot))}_{L^1(\RR^d)}^2+R^{-2}\norm{\nabla e^*(q(t,\cdot))}_{L^2(\RR^d)}^2 \Big)^{r/2}  \, dt.
 \]
 Once again optimizing over $R$, we have 
 \[
 \norm{e^*(q)}_{L^r([0,T];L^2(\RR^d))}^r \lesssim_d \int_0^T  \norm{e^*(q(t,\cdot)}_{L^1(\RR^d)}^{\frac{2r}{(d+2)}}\norm{\nabla e^*(q(t,\cdot))}_{L^2(\RR^d)}^{\frac{dr}{(d+2)}}\, dt
 \]
 \[
 \lesssim_d \norm{e^*(q)}_{L^1(Q_T)}^{\frac{2r}{d+2}}\norm{\nabla e^*(q)}_{L^{\frac{dr}{d+2-2r}}([0,T];L^2(\RR^d))}^{\frac{dr}{d+2}}.
 \]
 Thus, 
 \[
  \norm{e^*(q)}_{L^r([0,T];L^2(\RR^d))}\lesssim_d \norm{e^*(q)}_{L^1(Q_T)}^{\frac{2}{d+2}}\norm{\nabla e^*(q)}_{L^{\frac{dr}{d+2-2r}}([0,T];L^2(\RR^d))}^{\frac{d}{d+2}}.
 \]
 If we choose $r=\frac{2d+4}{d+4}$ we get
 \[
  \norm{e^*(q)}_{L^{\frac{2d+4}{d+4}}([0,T];L^2(\RR^d))}\lesssim_d \norm{e^*(q)}_{L^1(Q_T)}^{\frac{2}{d+2}}\norm{\nabla e^*(q)}_{L^2(Q_T)}^{\frac{d}{d+2}}.
 \]
 Finally, since $\nabla e^*(q)=\rho\nabla q$ by the chain rule for Sobolev functions, we have 
  \[
  \norm{e^*(q)}_{L^{\frac{2d+4}{d+4}}([0,T];L^2(\RR^d))}\lesssim_d \norm{\rho}_{L^{\infty}(Q_T)}^{\frac{d}{d+2}}\norm{e^*(q)}_{L^1(Q_T)}^{\frac{2}{d+2}}\norm{\nabla q}_{L^2(Q_T)}^{\frac{d}{d+2}}.
 \]
Fixing a compact set $K\subset \RR^d$, we also have 
 \[
  \norm{q}_{L^{\frac{2d+4}{d+4}}([0,T];L^2(K))}\leq \beta T|K|+\norm{\bar{q}}_{L^{\frac{2d+4}{d+4}}([0,T];L^2(Q_T))}\leq \beta T|K|+\beta\norm{e^*(q)}_{L^{\frac{2d+4}{d+4}}([0,T];L^2(\RR^d))}
 \]

\end{proof}

\section{Main results}
\label{sec:main}

At last, we are ready to combine our work to prove the main results of this paper.  We will begin by constructing solutions to the system (\ref{eq:system_q}) and then we will show that these can be converted into solutions to the original system (\ref{eq:system}).   

The construction of solutions to (\ref{eq:system_q}) is based on a vanishing viscosity approach. To that end, we consider a viscous analogue of system (\ref{eq:system_q}) where we add viscosity to both of the species $\rho_1, \rho_2$.  Given a viscosity parameter $\gamma\geq 0$, we introduce the system:
\begin{equation}\label{eq:viscous_system_q}
\begin{cases}
\partial_t \rho_1-\gamma\Delta\rho_1-\nabla \cdot(\frac{\rho_1}{\rho}\nabla q)+\nabla \cdot (\rho_1 V)=\rho_1 F_{1,1}\big((z^*)^{-1}(q),n\big)+\rho_2 F_{1,2}\big((z^*)^{-1}(q),n\big),\\
\partial_t \rho_2-\gamma\Delta \rho_2-\nabla \cdot(\frac{\rho_2}{\rho}\nabla q)+\nabla \cdot (\rho_2 V)=\rho_1 F_{2,1}\big((z^*)^{-1}(q),n\big)+\rho_2 F_{2,2}\big((z^*)^{-1}(q),n\big),\\
\rho q=e(\rho)+e^*(q),\\
\partial_t n-\alpha \Delta n=-n(c_1\rho_1+c_2\rho_2).
\end{cases}
\end{equation}
We define weak solutions to this system as follows.
\begin{definition}
Given a viscosity parameter $\gamma\geq 0$
and initial data $\rho_1^0, \rho_2^0\in X(e)$ and  $n^0\in L^2(\RR^d)$, we say that $(\rho_1, \rho_2, q, n)\in \cX(e)\times \cX(e)\times\cY(e^*)\times L^2_{\loc}([0,\infty);H^1(\RR^d))$ is a weak solution to the system (\ref{eq:viscous_system_q})
with initial data $(\rho_1^0, \rho_2^0, n^0)$,  if $\rho q=e(\rho)+e^*(q)$ almost everywhere, $\gamma\nabla \rho_1, \gamma\nabla \rho_2\in L^2_{\loc}([0,\infty);L^2(\RR^d))$, and for every test function $\psi\in H^1_c([0,\infty);H^1(\RR^d))$
\begin{equation}\label{eq:weak_1}
    \int_{\RR^d} \psi(0,x)\rho_{1}^0=\int_{Q_{\infty}}  \nabla \psi \cdot \big(\frac{\rho_{1}}{\rho}\nabla q+\gamma\nabla\rho_1-\rho_1V\big)-\rho_{1}\partial_t\psi -\psi\big(\rho_1 F_{1,1}\big((z^*)^{-1}(q),n\big)+\rho_2 F_{1,2}\big((z^*)^{-1}(q),n\big)\big),
\end{equation}
\begin{equation}\label{eq:weak_2}
\int_{\RR^d} \psi(0,x)\rho_{2}^0=\int_{Q_{\infty}}   \nabla \psi \cdot \big(\frac{\rho_{2}}{\rho}\nabla q+\gamma\nabla \rho_2-\rho_2 V\big)-\rho_{2}\partial_t\psi-\psi\big(\rho_1 F_{2,1}\big((z^*)^{-1}(q),n\big)+\rho_2 F_{2,2}\big((z^*)^{-1}(q),n\big)\big),\\
\end{equation}
\begin{equation}\label{eq:weak_3}
\int_{\RR^d}\psi(0,x)n^0=\int_{Q_{\infty}}  \alpha\nabla\psi\cdot\nabla n-n\partial_t \psi+n(c_1\rho_{1}+c_2\rho_{2})\psi
\end{equation}
where $\rho=\rho_1+\rho_2$.
\end{definition}

When $\gamma>0$, the existence of weak solutions to (\ref{eq:viscous_system_q}) is straightforward, as the individual densities will be bounded in $L^2_{\loc}([0,\infty);H^1(\RR^d))\cap H^1_{\loc}([0,\infty);H^{-1}(\RR^d))$.  Since this space is compact in $L^2_{\loc}([0,\infty);L^2(\RR^d))$, one can construct the solutions as limits of an even more regularized system (with enough regularity existence of solutions can be shown with a standard but tedious Picard iteration).   Thus, we can assume the existence of a sequence $(\rho_{1,k}, \rho_{2,k}, q_k, n_k)$ such that for each $k$ the variables are a weak solution to (\ref{eq:viscous_system_q}) with viscosity parameter $\gamma_k>0$.   We will then use our efforts from the past two sections to show that when $\gamma_k\to 0$ we can still pass to the limit in equations (\ref{eq:weak_1}-\ref{eq:weak_3}) to obtain a solution to (\ref{eq:system_q}).  In fact, we will show that we can pass to the limit even when the underlying energy function $e_k$ is changing along the sequence.

We begin with the strong precompactness for the pressure gradient.
\begin{prop}\label{prop:strong_convergence} Let $e_k$ be a sequence of energy functions satisfying (e1-e3) and suppose there exists an energy $e$ satisfying (e1-e3) such that $e_k$ converges pointwise everywhere to $e$.  Let $\rho_k\in \cX(e_k), q_k\in \cY(e_k^*)$, and $\mu_k\in L^{\infty}(\frac{1}{\rho_k})$ be sequences of densities, pressure, and growth terms that converge weakly in $L^1_{\loc}(Q_{\infty})$ to limits $\rho\in \cX(e), q\in\cY(e^*), \mu\in L^{\infty}(\frac{1}{\rho})$.  If $\rho_k q_k$ converges weakly in $L^1_{\loc}(Q_{\infty})$ to $\rho q$
and for every $\omega\in W^{1,\infty}_c([0,\infty))$ 
\begin{equation}\label{eq:sequence_edi}
     \int_{Q_{\infty}} -e_k(\rho_k)\partial_t \omega+\omega|\nabla q_k|^2+\omega e^*_k(q_k)\nabla \cdot V-\omega\mu_k q_k \leq \int_{\RR^d}\omega(0) e_k(\rho_k(0,x))\, dx,
\end{equation}
\begin{equation}\label{eq:limit_edi}
    \int_{\RR^d}\omega(0) e(\rho(0,x))\, dx\leq  \int_{Q_{\infty}} -e(\rho)\partial_t \omega+\omega|\nabla q|^2+\omega e^*(q)\nabla \cdot V-\omega\mu q, 
\end{equation}
and
\begin{equation}\label{eq:qmu_limit}
\limsup_{k\to\infty} \int_{\RR^d}\omega(0) e_k(\rho_k(0,x))+\int_{Q_{\infty}} \omega q_k\mu_k\leq \int_{\RR^d}\omega(0) e(\rho(0,x))+\int_{Q_{\infty}} \omega q\mu,
\end{equation}
then $\nabla q_k$ converges strongly in $L^2_{\loc}([0,\infty);L^2(\RR^d))$ to $\nabla q$.  
\end{prop}
\begin{proof}
  If we combine (\ref{eq:sequence_edi}), (\ref{eq:qmu_limit}) and (\ref{eq:limit_edi}), we get the string of inequalities
\[
\limsup_{k\to\infty}  \int_{Q_{\infty}} -e_k(\rho_k)\partial_t \omega+\omega|\nabla q_k|^2+\omega e^*_k(q_k)\nabla \cdot V
\]
\[
\leq \limsup_{k\to\infty} \int_{\RR^d}\omega(0) e(\rho_k(0,x))+\int_{Q_{\infty}} \omega q_k\mu_k\leq \int_{\RR^d}\omega(0) e(\rho(0,x))+\int_{Q_{\infty}}
\omega\mu q
\]
\[
\leq \int_{Q_{\infty}} -e(\rho)\partial_t \omega+\omega|\nabla q|^2+\omega e^*(q)\nabla \cdot V
\]
Thanks to Prop \ref{prop:e_e^*}, the weak convergence of $\rho_k q_k$ to $\rho q$ implies that $e_k(\rho_k), e^*_k(q_k)$ converge weakly in $L^1_{\loc}(Q_{\infty})$ to $e(\rho), e^*(q)$ respectively.  Therefore,
\begin{equation}\label{eq:H1_upper}
\limsup_{k\to\infty} \int_{Q_{\infty}} \omega|\nabla q_k|^2\leq \int_{Q_{\infty}}\omega|\nabla q|^2<\infty.
\end{equation}
The $L^2_{\loc}([0,\infty);L^2(\RR^d))$ boundedness of $\nabla q_k$ along with the weak $L^1_{\loc}(Q_{\infty})$ convergence of $q_k$ to $q$ implies that $\nabla q_k$ converges weakly in $L^2_{\loc}([0,\infty);L^2(\RR^d))$ to $\nabla q$.  Combining the weak convergence with the upper semicontinuity property (\ref{eq:H1_upper}), it now follows that $\nabla q_k$ converges strongly in $L^2_{\loc}([0,\infty);L^2(\RR^d))$ to $\nabla q$.
\end{proof}

The next two Lemmas are technical results that will help us guarantee that we can pass to the limit in all of the terms in (\ref{eq:weak_1}) and (\ref{eq:weak_2}).
\begin{lemma}\label{lem:chain_rule_convergence}
Let $e_k$ be a sequence of energies satisfying (e1-e3) and suppose there exists an energy $e$ satisfying (e1-e3) such that $e_k$ converges pointwise everywhere to $e$. Let $\rho_k\in \cX(e_k), q_k\in \cY(e^*_k)$ be sequences of uniformly bounded density and pressure variables that satisfy the duality relation $\rho_kq_k=e_k(\rho_k)+e^*_k(q_k)$ almost everywhere. 
If $q_k$ converges strongly in $L^2_{\loc}([0,\infty);\dot{H}^1(\RR^d))\cap L^{\frac{2d+4}{d+4}}_{\loc}(Q_{\infty})$ to a limit $q$ and $\rho_k$ converges weakly in $L^2_{\loc}([0,\infty);L^2(\RR^d))$ to a limit $\rho$, then 
\[
\limsup_{k\to\infty} \int_{D} |\rho-\rho_k||\nabla q|^2=0
\]
for any compact set $D\subset Q_{\infty}$
\end{lemma}
\begin{proof}
Clearly for any $\vp\in C^{\infty}_c(Q_{\infty})$ we have
\[
\limsup_{k\to\infty} \int_{Q_{\infty}} \vp \rho_k q_k= \int_{Q_{\infty}} \vp \rho q.
\]
Thus, by Proposition \ref{prop:e_e^*}, the limiting variables satisfy the duality relation $\rho q=e(\rho)+e^*(q)$ almost everywhere.

Let $M=\sup_k \norm{\rho_k}_{L^{\infty}(D)}<\infty$. Define $\bar{e}_k^*$ and $\bar{e}^*$ such that $\bar{e}_k^*(0)=0, \bar{e}^*(0)=0$, and
\[
\partial \bar{e}_k^*(b)=\{\min(a,M): a\in \partial e^*_k(b)\}, \quad \partial \bar{e}^*(b)=\{\min(a,M): a\in \partial e^*(b)\}
\]
Let $\bar{e}_k=(\bar{e}_k^*)^*$ and $\bar{e}=(\bar{e}^*)^*$.
Clearly, we still have the duality relations $\rho_k q_k=\bar{e}(\rho_k)+\bar{e}^*(q_k)$ and $\rho q=\bar{e}(\rho)+\bar{e}^*(q)$ almost everywhere.
It also follows that $\bar{e}_k^*, \bar{e}^*$ are uniformly Lipschitz on the entire real line and uniformly bounded on compact subsets of $\RR$.  As a result, $\bar{e}_k^*$ must converge uniformly on compact subsets of $\RR$ to $\bar{e}^*.$

Fix some $\delta>0$.  Convexity and the duality relation imply that
\[
\rho_k\leq \frac{\bar{e}^*_k(q_k+\delta)-\bar{e}^*_k(q_k)}{\delta}, \quad \rho\leq \frac{\bar{e}^*(q+\delta)-\bar{e}^*(q)}{\delta},
\]
and
\[
\rho_k\geq \frac{\bar{e}^*_k(q_k)-\bar{e}^*_k(q_k-\delta)}{\delta}, \quad \rho\geq \frac{\bar{e}^*_k(q)-\bar{e}^*(q-\delta)}{\delta}.
\]
Therefore, 
\[
\int_{D} |\rho-\rho_k||\nabla q|^2
\]
\[
\leq \int_{D} \Big(|\frac{\bar{e}^*_k(q_k+\delta)+\bar{e}^*(q-\delta)-\bar{e}^*_k(q_k)-\bar{e}^*(q)}{\delta}|+|\frac{\bar{e}^*(q+\delta)+\bar{e}_k^*(q_k-\delta)-\bar{e}^*_k(q_k)-\bar{e}^*(q)}{\delta}| \Big)|\nabla q|^2.
\]
Thus, it follows that 
\[
\limsup_{k\to\infty}\int_{D} |\rho-\rho_k||\nabla q|^2\leq 2\int_{D} |\frac{\bar{e}^*(q+\delta)+\bar{e}^*(q-\delta)-2\bar{e}^*(q)}{\delta}||\nabla q|^2
\]

If $\bar{e}^*$ is continuously differentiable at a point $b\in \RR$, then
\[
\lim_{\delta\to 0} \frac{\bar{e}^*(b+\delta)+\bar{e}^*(b-\delta)-2\bar{e}^*(b)}{\delta}=0.
\]
The singular set $S\subset \RR$ of values where $\bar{e}^*$ is not continuously differentiable is at most countable.  Therefore, $|\nabla q|$ is zero almost everywhere on the set $\{(t,x)\in D: q(t,x)\in S\}$.  Hence, by dominated convergence,
\[
\lim_{\delta\to 0} 2\int_{D} |\frac{\bar{e}^*(q+\delta)+\bar{e}^*(q-\delta)-2\bar{e}^*(q)}{\delta}||\nabla q|^2=0.
\]
\end{proof}

\begin{lemma}\label{lem:source_limit}
Let $z_k$ be a sequence of energies satisfying (z1-z3) and suppose there exists an energy $z$ satisfying (z1-z3) such that $z_k$ converges pointwise everywhere to $z$.  Define $e_k, e$ by formula (\ref{eq:e_def}).
Suppose that $(\rho_{1,k}, \rho_{2,k}, q_k,n_k)\in \cX(e_k)\times \cX(e_k)\times \cY(e^*_k)\times L^2_{\loc}([0,\infty);H^1(\RR^d))$ is a sequence such that $(\rho_{1,k}+\rho_{2,k})q_k=e_k(\rho_{1,k}+\rho_{2,k})+e^*_k(q_k)$ almost everywhere.    Suppose that $\rho_{1,k}, \rho_{2,k}$ converge weakly in $L^r_{\loc}([0,\infty);L^r(\RR^d)$ to limits $\rho_{1}, \rho_2\in \cX(e)$, $q_k$ converges strongly in $L^{\frac{2d+4}{d+4}}_{\loc}([0,\infty);L^2_{\loc}(\RR^d)) \cap L^2_{\loc}([0,\infty);\dot{H}^{1}(\RR^d))$ to a limit $q$, and $n_k$ converges strongly in $L^2_{\loc}([0,\infty);L^2(\RR^d))$ to a limit $n$. If the growth terms $F_{i,j}$ satisfy assumptions (F1-F2), then $\rho_{j,k}F_{i,j}\big(z_k^{-1}(q_k),n_k\big)$ converges weakly in $L^r_{\loc}([0,\infty);L^r(\RR^d))$ to $\rho_{j}F_{i,j}\big(z^{-1}(q),n)\big)$ for all $i,j\in \{1,2\}$ and any $r<\infty$.
\end{lemma}
\begin{proof}
It suffices to prove the convergence of $\rho_{1,k}F_{1,1}\big(z_k^{-1}(q_k),n_k\big)$ to $\rho_{1}F_{1,1}\big(z^{-1}(q),n\big)$, the argument for the other terms is identical. Let $\vp\in C_c^{\infty}(Q_{\infty})$ and let $D\subset Q_{\infty}$ be a compact set containing the support of $\vp$. 
For $N\in \RR$ define
$S_{k,N}:=\{(t,x)\in D:  q_k(t,x)+n_k(t,x)>N\}.$
From the uniform bounds on the norms of $q_k, n_k$ it follows that $ \lim_{N\to\infty} \sup_k |S_{k,N}|=0.$
Thus, we can assume without loss of generality that $q_k, n_k$ are uniformly bounded by some $M>0$ (and of course this same logic applies to $q, n$ as well).   

Let $b_{\infty}=\sup\{b\in\RR: z^*(b)<\infty\}$.
Fix $\epsilon\in (0, z^*(b_{\infty})/2)$ and let
$q_{k,\epsilon}=\min(\max(\epsilon, q_k), z^*(b_{\infty})-\epsilon), q_{\epsilon}=\min(\max(\epsilon, q), z^*(b_{\infty})-\epsilon)$.
It now follows that $(z_k^*)^{-1}(q_{k,\epsilon}), (z^*)^{-1}(q_{\epsilon})$ are uniformly bounded in $L^{\infty}(D)$.  Thanks to Lemma \ref{lem:convex_convergence}, we know that $(z_k^*)^{-1}$ converges uniformly to $(z^*)^{-1}$ on $(\epsilon, z^*(b_{\infty})-\epsilon)$.
Combining this with properties (F1-F2), and the various convergence properties of $q_k, n_k, \rho_{1,k}$ it follows that 
\[
\limsup_{k\to\infty}\Big|\int_{Q_{\infty}} \vp \Big(\rho_{1,k}F_{1,1}\big((z^*_k)^{-1}(q_{k,\epsilon}),n_k\big)-\rho_{1}F_{1,1}\big((z^*)^{-1}(q_{\epsilon}),n\big)\Big)\Big|=0.
\]
Thus, it remains to show that
\begin{equation}\label{eq:source_vanish}
\lim_{\epsilon\to 0^+} \Big|\int_{Q_{\infty}} \vp \rho_{1}\Big(F_{1,1}\big((z^*)^{-1}(q_{\epsilon}),n\big)-F_{1,1}\big((z^*)^{-1}(q),n\big)\Big)\Big|=0
\end{equation}
and
\begin{equation}\label{eq:source_k_vanish}
\lim_{\epsilon\to 0^+}\limsup_{k\to\infty} \Big|\int_{Q_{\infty}} \vp \rho_{1,k}\Big(F_{1,1}\big((z^*_k)^{-1}(q_{k,\epsilon}),n_k\big)-F_{1,1}\big((z^*_k)^{-1}(q_{k}),n_k\big)\Big)\Big|=0.
\end{equation}

To do this we will exploit the density pressure duality relationship. Thanks to the relationship between $e$ and $z$, we can express the duality relation as $(\rho_{1,k}+\rho_{2,k})(z^*_k)^{-1}(q_k)=z_k(\rho_{1,k}+\rho_{2,k})+q_k $. Fix some $\delta>0$ and split the support of $\rho_{1,k}$ into the sets $\rho_{1,k}<\delta$ and $\rho_{1,k}\geq \delta$.  Again using duality, we have 
\[
0\leq \rho_{1,k}\leq \rho_{1,k}+\rho_{2,k}\in \partial z_k^*\circ (z_k^*)^{-1}\circ q_k
\]
Thus, for almost every $(t,x)$  where $\rho_{1,k}(t,x)\geq \delta$, it follows that $(z_k^*)^{-1}$ is at worst $\delta^{-1}$ Lipschitz at the value $q_k(t,x)$ and $(z_k^*)^{-1}(q_k(t,x))$ is uniformly bounded with respect to $k$.  Thus,   
\[
 \Big|\int_{Q_{\infty}} \vp \rho_{1,k}\Big(F_{1,1}\big((z^*_k)^{-1}(q_{k,\epsilon}),n_k\big)-F_{1,1}\big((z^*_k)^{-1}(q_{k}),n_k\big)\Big)\Big|
\]
\[
\leq 
  B\delta\norm{\vp}_{L^1(D)}+\omega_{\delta}(2\epsilon\delta^{-1})\norm{\rho_{1,k}}_{L^1(D)}\norm{\vp}_{L^{\infty}(D)}+\norm{\rho_{1,k}\vp}_{L^{\infty}(D)}|D_{k,\epsilon}|
\]
where $B$ is a bound on $F_{1,1}$ and $\omega_{\delta}$ is the modulus of continuity of $F_{1,1}$ on the bounded set $\Big(\bigcup_{k}\{(z_k^*)^{-1}(q_k(t,x)): \rho_{1,k}(t,x)\geq \delta\}\Big)\times [0,M]$ and $D_{k,\epsilon}=\{(t,x)\in D: q_k(t,x)>z^*(b_{\infty})+\epsilon\}$. The convergence of $z_k$ to $z$ implies that $\limsup_{k\to\infty} |D_{k,\epsilon}|=0$ for all fixed $\epsilon>0$.   Thus, sending $k\to\infty$, then $\epsilon\to 0^+$, and then $\delta\to 0^+$, we get (\ref{eq:source_k_vanish}).  The strong convergence of $q_k$ implies that the duality relation $(\rho_1+\rho_2)(z^*)^{-1}(q)=z(\rho_1+\rho_2)+q$ holds, thus we can use a similar argument to obtain (\ref{eq:source_vanish}).

\end{proof}

At last, we are ready to prove our main result, which  will let us pass to the limit when we consider sequences of weak solutions to (\ref{eq:viscous_system_q}).  Note that the following theorem applies in the case where the viscosity is decreasing to zero along the sequence, as well as when the viscosity is zero along the entire sequence.
\begin{theorem}\label{thm:main}
Let $z_k$ be a sequence of energies satisfying (z1-z3). Suppose there exists an energy $z$ satisfying (z1-z3) such that $z_k$ converges pointwise everywhere to $z$.  Define $e_k, e$ by formula (\ref{eq:e_def}).  Let $\rho_1^0, \rho_2^0\in L^1(\RR^d)\cap L^{\infty}(\RR^d), n^0\in L^2(\RR^d)$ be initial data such that $e(\rho_1^0+\rho_2^0)\in L^1(\RR^d)$.  Let $V\in L^2_{\loc}([0,\infty);L^2(\RR^d))$ be a vector field such that $\nabla \cdot V\in L^{\infty}(Q_{\infty})$ and let $F_{i,j}$ be source terms satisfying (F1-F2).   Let $\rho_{1,k}, \rho_{2,k}\in \cX(e_k)$, $q_k\in \cY(e^*_k)$, $n_k\in L^2_{\loc}([0,\infty);H^1(\RR^d))$ be sequences of density pressure and nutrient variables such that $\nabla \rho_{1,k}, \nabla \rho_{2,k}\in L^2_{\loc}([0,\infty);L^2(\RR^d))$.  Suppose that for each $k$, the variables $(\rho_{1,k}, \rho_{2,k}, q_k, n_k)$ are weak solutions to the system (\ref{eq:viscous_system_q}) with energy $e_k$, viscosity constant $\gamma_k\geq 0$, and initial data $(\rho_1^0, \rho_2^0, n^0)$. 
 If $\gamma_k$ converges to $0$ and at least one of the following two conditions hold: 
 \begin{enumerate}[(a)]
     \item \label{sc:a} $\partial z(a)$ is a singleton for all $a\in (0,\infty)$,
     \item \label{sc:b} the source terms satisfy the additional condition (F3),
 \end{enumerate}
 then any limit point $(\rho_1, \rho_2, q,n)$ of the sequence is a solution of (\ref{eq:system_q}).
\end{theorem}
\begin{proof}
\textit{Step 1: Uniform bounds, basic convergence properties, and parabolic structure.}

Summing the first two equations of (\ref{eq:viscous_system_q}) together, we see that for any test function $\psi\in W^{1,1}_c([0,\infty);H^1(\RR^d))$ $\rho_k, q_k$ are weak solutions to the parabolic equation
\begin{equation}\label{eq:parabolic_k_1}
\int_{\RR^d} \psi(0,x)\rho^0=\int_{Q_{\infty}}  -\rho_{k}\partial_t\psi +\nabla \psi \cdot (\nabla q_k+\gamma_k\nabla \rho_k)-\rho_k\nabla\psi  \cdot V-\psi\mu_k
\end{equation}
where $\rho_k=\rho_{1,k}+\rho_{2,k}$,  $\mu_k=\mu_{1,k}+\mu_{2,k}$ and $\mu_{i,k}=\rho_{1,k} F_{i,1}\big((z_k^*)^{-1}(q_k,n_k)\big)+\rho_{2,k} F_{i,2}\big((z_k^*)^{-1}(q_k,n_k)\big)$.  

Thanks to Proposition \ref{prop:estimates}, $\rho_k, q_k, \mu_k$ must satisfy the energy dissipation inequality \[
 \int_{Q_{\infty}} -e(\rho_k)\partial_t \omega+\omega|\nabla q_k|^2+\omega e^*(q_k)\nabla \cdot V-\omega\mu_k q_k \leq \int_{\RR^d}\omega(0) e(\rho^0(x))\, dx,
\]
for every nonnegative $\omega\in W^{1,\infty}([0,\infty))$ and the estimates
(\ref{eq:gamma_nabla_rho})-(\ref{eq:rho_p_extra_control}). After plugging estimate (\ref{eq:gamma_nabla_rho}) into estimate (\ref{eq:rho_h_minus1}), it follows that all of the estimates (\ref{eq:l1_growth}-\ref{eq:rho_p_extra_control}) are independent of $k$ and only depend on $\rho^0$, $V$ and the bounds on $F_{i,j}$. Thus, $\rho_k, q_k$ are uniformly bounded in the norms estimated in (\ref{eq:l1_growth})-(\ref{eq:rho_p_extra_control}). As a result, there must exist $\rho\in \cX(e)$, $q\in \cY(e^*)$ and $\mu\in L^{\infty}_{\loc}([0,\infty); L^{\infty}(\RR^d)\cap L^1(\RR^d))$ such that $\rho_k, q_k,\mu_k$ converge weakly in $L^{\frac{2d+4}{d+4}}_{\loc}([0,\infty);L^2_{\loc}(\RR^d))$ (along a subsequence that we do not relabel) to $\rho, q, \mu$ respectively.  Note that for $\rho_k, \mu_k$ the weak convergence in fact holds in $L^{r}_{\loc}(Q_{\infty})$ for any $r<\infty$.

Property (F2) implies that $0\leq \rho_{1,k}, \rho_{2,k}\leq \rho_k$.  Hence, $\rho_{1,k}, \rho_{2,k}$ are uniformly bounded in $L^{\infty}_{\loc}([0,\infty);L^1(\RR^d)\cap L^{\infty}(\RR^d))$ and there exist limit points $\rho_1, \rho_2$ (and a subsequence that we do not relabel) such that $\rho_{1,k}, \rho_{2,k}$ converge weakly in $L^{r}_{\loc}([0,\infty);L^1(\RR^d)\cap L^{r}(\RR^d))$ to $\rho_1, \rho_2$ respectively for any $r<\infty$.  Furthermore, the bounds on $\rho_{1,k}, \rho_{2,k}$ combined with standard results for the heat equation imply that $n_k$ is uniformly bounded in $L^2_{\loc}([0,\infty);H^1(\RR^d))\cap H^1_{\loc}([0,\infty);H^{-1}(\RR^d))$.  Hence, the Aubin-Lions Lemma implies that there exists a limit point $n\in L^2_{\loc}([0,\infty);H^1(\RR^d))$ and a subsequence (that we do not relabel) such that $n_k$ converges to $n$ in $L^2_{\loc}([0,\infty);L^2(\RR^d))$.

Thanks to the linear structure of equation (\ref{eq:parabolic_k_1}), the convergence properties we have established are strong enough to send $k\to\infty$. Thus, $\rho, q, \mu$ satisfy the weak equation
\begin{equation}\label{eq:precompact_weak_limit_eq}
\int_{\RR^d} \psi(0,x) \rho^0(x)\, dx= \int_{Q_{\infty}} \nabla q\cdot \nabla \psi-\rho\partial_t\psi-\rho V\cdot\nabla \psi-\mu \psi.
\end{equation}
for any $\psi\in W^{1,1}_c([0,\infty);H^1(\RR^d))$
After taking the limit, the bounds on $\rho, q, \mu$ inherited from the estimates (\ref{eq:l1_growth}-\ref{eq:rho_p_extra_control}) allow us to conclude that (\ref{eq:precompact_weak_limit_eq}) holds for any  $\psi\in W^{1,1}_c([0,\infty);L^1(\rho)\cap \dot{H}^1(\RR^d))$. Thus, Proposition \ref{prop:edr} implies that  for every $\omega\in W^{1,\infty}_c([0,\infty))$ the limit variables $\rho, \mu, q$ satisfy the energy dissipation relation
\[
 \int_{\RR^d}\omega(0) e(\rho(0,x))\, dx=  \int_{Q_{\infty}} -e(\rho)\partial_t \omega+\omega|\nabla q|^2+\omega e^*(q)\nabla \cdot V-\omega\mu q.
\]

\noindent\textit{Step 2: Weak convergence of the products $\rho_{1,k}q_k, \rho_{2,k}q_k$.}

We want to use Lemma \ref{lem:spacetime_cc} to prove that $\rho_{i,k}q_k$ converges weakly  to $\rho_i q$ for $i=1,2$.   This will imply that $\rho_k q_k$ converges weakly to $\rho q$.   Fix some $\epsilon>0$ and let $\eta_{\epsilon}$ be a spatial mollifier.  Define $\rho_{i,k,\epsilon}=\eta_{\epsilon}*\rho_{i,k}$ and $\rho_{i,\epsilon}=\eta_{\e}*\rho_i$.
Thanks to estimates  (\ref{eq:l1_growth}-\ref{eq:rho_linf}), it follows that
 \[
 \sup_k\;\; \norm{\partial_t \rho_{i,k,\epsilon}}_{L^2(Q_T)}+\norm{\nabla \rho_{i,k,\epsilon}}_{L^2(Q_T)}\lesssim_{\epsilon} \sup_k\; \norm{\rho_{i,k}}_{L^2(Q_T)}+\norm{\rho_{i,k}}_{H^1([0,T];H^{-1}(\RR^d))}<\infty.
 \]
 Thus, for $\epsilon>0$ fixed, $\rho_{i,k,\epsilon}$ is uniformly equicontinuous in $L^2(Q_T)$.  The uniform bounds (\ref{eq:l1_growth}) and (\ref{eq:rho_linf}) automatically upgrade this to uniform equicontinuity in  $L^r(Q_T)\cap L^1(Q_T)$ for any $r<\infty$. In addition, the estimates (\ref{eq:rho_p_extra_control}) and (\ref{eq:nabla_q_control}) imply that $q_k$ is spatially equicontinuous in $L^{\frac{2d+4}{d+4}}_{\loc}(Q_{\infty})$. Thus, we can apply Lemma \ref{lem:spacetime_cc} to conclude that $\rho_{i,k}q_k$ converges weakly in $(C_c(Q_{\infty}))^*$ to $\rho_i q$ for $i=1,2$.  The uniform boundedness of $\rho_{i,k} q_k$ in $L_{\loc}^{\frac{2d+4}{d+4}}([0,\infty);L^2(\RR^d))$ gives us the automatic upgrade to weak convergence in $L_{\loc}^{\frac{2d+4}{d+4}}([0,\infty);L^2(\RR^d))$.  Now Proposition \ref{prop:e_e^*} implies that $\rho q=e(\rho)+e^*(q)$ almost everywhere and $e(\rho_k)$ and $e^*(q_k)$ converge weakly to $e(\rho)$ and $e^*(q)$ respectively.

 \noindent\textit{Step 3: Strong convergence of $\nabla q_k$ to $\nabla q$ in $L^2_{\loc}([0,\infty);L^2(\RR^d))$.}
 
 We now want to use Proposition \ref{prop:strong_convergence} to prove the strong convergence of the pressure gradient.  Note that the pointwise everywhere convergence of $z_k$ to $z$ implies the pointwise everywhere convergence of $e_k$ to $e$. 
 We have already shown that $\rho_k q_k$ converges weakly to $\rho q$ and verified the inequalities (\ref{eq:sequence_edi}) and (\ref{eq:limit_edi}).  Thus it remains to show that the upper semicontinuity property (\ref{eq:qmu_limit}) holds.  To verify this condition, we will need to consider the scenarios (\ref{sc:a}) and (\ref{sc:b}) separately.
 
 \noindent\textit{Step 3a: Scenario (\ref{sc:a}) holds.}
 When $\partial z(a)$ is a singleton for all $a\in (0,\infty)$, it follows that $\partial e(a)$ is a singleton for all $a\in (0,\infty)$ and hence $e^*$ must be strictly convex on $(0,\infty)\cap (e^*)^{-1}(\RR)$.
Thus,  Lemma \ref{lem:in_measure} implies that $q_k$ converges in measure to $q$.  Since $q_k$ is uniformly bounded in $L^{\frac{2d+4}{d+4}}_{\loc}([0,\infty);L^2_{\loc}(\RR^d))$, we can upgrade the convergence in measure to strong convergence  in $L^{r}_{\loc}(Q_{\infty})$  for any $r<\frac{2d+4}{d+4}$.
 From the strong convergence, it is automatic that 
\[
\limsup_{k\to\infty}\int_{Q_{\infty}} \omega\mu_kq_k= \int_{Q_{\infty}} \omega \mu q
\]
 any $\omega\in W^{1,\infty}_c([0,\infty))$.

\noindent\textit{Step 3b: Scenario (\ref{sc:b}) holds}

Without strict convexity of the dual energy, the weak convergence of $e^*_k(q_k)$ does not give us strong convergence of $q_k$. Thus, to prove (\ref{eq:qmu_limit}) we will need a more delicate argument that exploits the structure of the product $q_k\mu_k$

We begin by fixing some $\delta>0$ and letting $J_{\delta}$ be a space time mollifier.  Set $q_{k,\delta}:=J_{\delta}*q_k$ and $q_{\delta}:=q*J_{\delta}$.  It is clear that $q_{k,\delta}$ converges strongly to $q_{\delta}$ in $L^{2}_{\loc}([0,\infty);L^2_{\loc}(\RR^d))$ and $q_{\delta}$ converges strongly to $q$ in $L^{\frac{2d+4}{d+4}}_{\loc}([0,\infty);L^2_{\loc}(\RR^d))$. Thus, it will be enough to show that
 \[
 \liminf_{\delta\to 0}\limsup_{k\to\infty} \int_{Q_{\infty}} \omega (q_{k}-q_{k,\delta})\mu_{i,k}\leq 0,
 \]
 for $i=1,2$.
 
We focus on the case $i=1$ (the argument for $i=2$ is identical).  Assumption (F3) and the monotonicity of $(z_k^*)^{-1}$ guarantees that $q\mapsto F_{1,1}\big((z^*_k)^{-1}(q),n\big)+F_{1,2}\big((z^*_k)^{-1}(q),n\big)$ is decreasing for each fixed value of $n$. As a result, there must exist a function $f_k:[0,\infty)\times [0,\infty)\to \RR$ such that for each fixed value of $n$, we have $f_k(0,n)=0$, $q\mapsto f_k(q,n)$ is convex, and  $-\partial_q f_k(q,n)=F_{1,1}\big((z^*_k)^{-1}(q),n\big)+F_{1,2}\big((z^*_k)^{-1}(q),n\big)$. The structure of $\mu_{1,k}$ combined with the convexity of $f_k$ implies that
\[
\int_{Q_{\infty}} \omega (q_{k}-q_{k,\delta})\mu_{i,k}\leq \int_{Q_{\infty}} \omega \rho_{1,k}\big( f_k(q_{k,\delta},n_k)-f_k(q_k,n_k)).
\]

Since $F_{1,1}+F_{1,2}$ is uniformly bounded over $\RR\times [0,\infty)$, it follows that $f_k$ is uniformly Lipschitz in the first argument. Uniform equicontinuity in the second argument is clear when $q=0$.   For $q>0$, fix some $\epsilon\in (0,q)$  and consider $n_1, n_2\geq 0$.  We see that
\[
|f_k(q,n_1)-f_k(q,n_2)|\leq \sum_{i=1}^2\int_0^q  |F_{1,i}\big((z_k^*)^{-1}(a),n_1\big)-F_{1,i}\big((z_k^*)^{-1}(a),n_2\big)|da.
\]
\[
\leq 2B\epsilon+q\sup_{b\in [(z_k^*)^{-1}(\epsilon),(z^*_k)^{-1}(q)]}\sum_{i=1}^2|F_{1,i}(b,n_1\big)-F_{1,i}\big(b,n_2\big)|, 
\]
where $B$ is a bound on $F_{1,1}+F_{1,2}$. Assumption (z3) and the pointwise everywhere convergence of $z_k$ to $z$ implies that  $(z_k^*)^{-1}(\epsilon),(z^*_k)^{-1}(q)$ are uniformly bounded with respect to $k$.  Thus, it now follows that $f_k$ is uniformly equicontinuous in the second argument on compact subsets of $[0,\infty)^2$.  As a result, $f_k$ must converge uniformly on compact subsets of $[0,\infty)^2$ to a limit function $f$ that is convex in the first variable and continuous in the second.

For all $k$ we have $|f_k(q,n)|\leq Bq$.  Thus,  it is now clear that 
 \[
\liminf_{\delta\to 0}\limsup_{k\to\infty} \int_{Q_{\infty}} \omega \rho_{1,k}\Big(| f_k(q_{k,\delta},n_k)-f(q,n)|+| f_k(q_{k},n_k)-f(q_k,n_k)|+ |f(q_k,n)-f(q_k,n_k)|\Big)=0.
\]
It remains to prove that
 \[
\limsup_{k\to\infty} \int_{Q_{\infty}} \omega \rho_{1,k}\big( f(q,n)-f(q_k,n))\leq 0.
\]

Let $f^*(a,n)=\sup_{q\in [0,\infty)} aq-f(q,n)$.  Given any smooth function $\psi\in C^{\infty}_c(Q_{\infty})$, we have 
 \[
 \int_{Q_{\infty}} \omega \rho_{1,k}\big( f(q,n)-f(q_k,n))\leq \int_{Q_{\infty}} \omega \rho_{1,k}\big( f(q,n)-q_k \psi)+\omega\rho_{1,k} f^*(\psi,n).
 \]
 Using the weak convergence of the product $\rho_{1,k}q_k$ to $\rho_1 q$ we see that
 \[
\limsup_{k\to\infty} \int_{Q_{\infty}} \omega \rho_{1,k}\big( f(q,n)-q_k \psi)+\rho_{1,k} f^*(\psi,n)=\int_{Q_{\infty}} \omega \rho_{1}\big( f(q,n)-q \psi)+\omega\rho_{1} f^*(\psi,n).
 \]
 Taking an infimum over $\psi$, we get 
 \[
\limsup_{k\to\infty} \int_{Q_{\infty}} \omega \rho_{1,k}\big( f(q,n)-f(q_k,n))\leq 0.
\]
as desired.

\noindent\textit{Step 4: Passing to the limit in the weak equations}

Now that we have obtained the strong convergence of the pressure gradient, we are ready to pass to the limit in the weak equations.  In Lemma \ref{lem:source_limit}, we showed that the source terms converge weakly to the desired limit under the convergence properties that we have established.   The weak convergence of the remaining terms is clear  except for the weak convergence of the product $\frac{\rho_{i,k}}{\rho_k}\nabla q_k$ to $\frac{\rho_{i}}{\rho}\nabla q$.  
Given some $\delta>0$, it follows from Lemma \ref{lem:chain_rule_convergence} that  $\frac{1}{\rho_k+\delta}\nabla q_k$ converges strongly in $L^2_{\loc}([0,\infty);L^2(\RR^d))$ to $\frac{1}{\rho+\delta}\nabla q$.  Thus, if we can show that
\begin{equation}\label{eq:delta_rho}
\liminf_{\delta\to 0}\Big( \int_{Q_T} \frac{\delta \rho_{i}}{\rho(\rho+\delta)}|\nabla q|^2+\limsup_{k\to\infty} \int_{Q_T} \frac{\delta \rho_{i,k}}{\rho_k(\rho_k+\delta)}|\nabla q_k|^2\Big)=0,
\end{equation}
then it will follow that $\frac{\rho_{i,k}}{\rho_k}\nabla q_k$  converges weakly in $L^2_{\loc}([0,\infty);L^2(\RR^d))$ to $\frac{\rho_{i}}{\rho}\nabla q$.

Since $\rho_{i,k}\leq \rho_k$ and $\rho_i\leq \rho$, the left hand side of (\ref{eq:delta_rho}) is bounded above by
\[
\liminf_{\delta\to 0}\Big( \int_{Q_T} \frac{\delta}{\rho+\delta}|\nabla q|^2+\limsup_{k\to\infty} \int_{Q_T} \frac{\delta}{\rho_k+\delta}|\nabla q_k|^2\Big)
\]
\[
=\liminf_{\delta\to 0} \int_{Q_T} \frac{2\delta}{\rho+\delta}|\nabla q|^2,
\]
where we have used Lemma \ref{lem:chain_rule_convergence} to go from the first line to the second.  The property $\limsup_{a\to 0^+}\frac{e(a)}{a}=0$ combined with the duality relation implies that $q=0$ whenever $\rho=0$.  As a result, $|\nabla q|$ gives no mass to the set of points where $\rho=0$. By dominated convergence
\[
\liminf_{\delta\to 0} \int_{Q_T} \frac{2\delta}{\rho+\delta}|\nabla q|^2=0.
\]

\end{proof}

\begin{cor}\label{cor:existence_strict}
Let $e$ be an energy satisfying (e1-e3) such that $\partial e(a)$ is a singleton for all $a\in (0,\infty)$.  Let $F_{i,j}$ be source terms satisfying (F1-F2).   Given initial data $\rho_1^0, \rho_2^0\in L^1(\RR^d)\cap L^{\infty}(\RR^d), n^0\in L^2(\RR^d)$ such that $e(\rho_1^0+\rho_2^0)\in L^1(\RR^d)$,
 there exists a weak solution $(\rho_1, \rho_2, q, n)\in \cX(e)\times \cX(e)\times \cY(e^*)\times L^2_{\loc}([0,\infty);H^1(\RR^d))$ to the system (\ref{eq:system_q}).
\end{cor}
\begin{proof}
For $\gamma_k=\frac{1}{k}$, the existence of a solution to the system (\ref{eq:viscous_system_q}) for the fixed energy $e$ is straightforward.  Using these solutions, we can pass to the limit as $k\to\infty$ using Theorem \ref{thm:main}. 
\end{proof}

\begin{cor}\label{cor:existence_G4}
Let $e$ be an energy satisfying (e1-e3) and let $F_{i,j}$ be source terms satisfying (F1-F3).   Given initial data $\rho_1^0, \rho_2^0\in L^1(\RR^d)\cap L^{\infty}(\RR^d), n^0\in L^2(\RR^d)$ such that $e(\rho_1^0+\rho_2^0)\in L^1(\RR^d)$,
 there exists a weak solution $(\rho_1, \rho_2, q, n)\in \cX(e)\times \cX(e)\times \cY(e^*)\times L^2_{\loc}([0,\infty);H^1(\RR^d))$ to the system (\ref{eq:system}).
\end{cor}
\begin{proof}
See Corollary \ref{cor:existence_strict}.
\end{proof}

\begin{cor}\label{cor:incompressible_limit}
Let $F_{i,j}$ be source terms satisfying (F1-F3).   Given initial data $\rho_1^0, \rho_2^0\in L^1(\RR^d)\cap L^{\infty}(\RR^d), n^0\in L^2(\RR^d)$ such that $\rho_1^0+\rho_2^0\leq 1$ almost everywhere, let $(\rho_{1,m}, \rho_{2,m}, q_m, n_m)\in \cX(e)\times \cX(e)\times \cY(e^*)\times L^2_{\loc}([0,\infty);H^1(\RR^d))$ be weak solutions of the system (\ref{eq:system_q}) with the energy $e_m(a)=\frac{1}{m}a^m$.  As $m\to\infty$, any weak limit point of the sequence $(\rho_{1,m}, \rho_{2,m}, q_m, n_m)$ is a solution to the system (\ref{eq:system_q}) with the incompressible energy
\[
e_{\infty}(a)=
\begin{cases}
0 & \textup{if} \;\; a\in [0,1],\\
+\infty & \textup{otherwise.}
\end{cases}
\]
\end{cor}
\begin{proof}
It is clear that $e_m$ converges pointwise everywhere to $e_{\infty}$. We can use Corollary \ref{cor:existence_strict} to construct weak solutions of (\ref{eq:system_q}) for each $m>0$.  We can then use Theorem \ref{thm:main} to pass to the limit $m\to\infty$.
\end{proof} 

At last, we will show that weak solutions to (\ref{eq:system_q}) can be easily converted into weak solutions to (\ref{eq:system}).
\begin{prop}
Let $z$ be an energy satisfying (z1-z3) and define $e$ by formula (\ref{eq:e_def}).  Suppose that $\rho_1^0, \rho_2^0\in L^1(\RR^d)\cap L^{\infty}(\RR^d), n^0\in L^2(\RR^d)$ is initial data such that $e(\rho_1^0+\rho_2^0), z(\rho_1^0+\rho_2^0)\in L^1(\RR^d)$.  If $(\rho_1, \rho_2, q, n)\in \cX(e)\times \cX(e)\times \cY(e^*)\times L^2_{\loc}([0,\infty);H^1(\RR^d))$ is a weak solution to the system (\ref{eq:system_q}) and we set $p=(z^*)^{-1}(q)$, then $(\rho_1, \rho_2, p, n)\in \cX(e)\times \cX(e)\times L^{\frac{2d+4}{d+4}}_{\loc}([0,\infty);L^1_{\loc}(\rho)) \cap L^2_{\loc}([0,\infty);\dot{H}^{1}(\rho))\times L^2_{\loc}([0,\infty);H^1(\RR^d))$ is a weak solution of (\ref{eq:system}).
\end{prop}
\begin{proof}
The duality relation $\rho q=e(\rho)+e^*(q)$ is equivalent to $p\rho=z(\rho)+z^*(p)=z(\rho)+q$. 
Given a compact subset $D\subset Q_{\infty}$ we have
\[
\int_D \rho|p|\leq \int_D |z(\rho)|+q
\]
Thus, $p\in L^1_{\loc}(\rho)$.

If $\sup\partial e^*(0)>0$, then $(z^*)^{-1}$ is uniformly Lipschitz on all of $[0,\infty)$ and $\rho$ is bounded away from zero on $q>0$.  By the duality relation and the chain rule for Sobolev functions, we have $\nabla p=\frac{1}{\rho}\nabla q$ and $\nabla p\in L^2_{\loc}([0,\infty);L^2(\RR^d))$.  In this case, it is now clear that $(\rho_1, \rho_2, p, n)$ is a weak solution to (\ref{eq:system}).

Otherwise, we are in the case where $q=0$ implies that $\rho=0$ and we cannot extend $(z^*)^{-1}$ to be uniformly Lipschitz on $[0,\infty)$.
Fix some $\delta>0$ and  let $\eta_{\delta}:[0,\infty)\to \RR$ be a smooth increasing function such that $\eta_{\delta}(a)=0$ if $a\leq \delta$ and $\eta_{\delta}(a)=1$ if $a\geq 2\delta$.
 Since $\limsup_{a\to 0^+}\frac{e(a)}{a}=0$, it follows that $\frac{1}{\rho}$ is bounded on $q\geq \delta$.  Given any test function $\vp\in L^{\infty}_c([0,\infty);W^{1,\infty}_c(\RR^d))$ we see that
\[
\int_{Q_{\infty}} p\nabla\cdot (\vp\eta_{\delta}(q))=\int_{Q_{\infty}} p\rho \frac{\eta_{\delta}(q)}{\rho}\nabla\cdot \vp+p\rho\frac{\eta_{\delta}'(q)}{\rho}\nabla q\cdot \vp
\]
Since $p$ must be bounded on the support of $\eta_{\delta}'(q)$, it follows that the above integral is well defined.  Define
 $q_{\delta}:=\max(q,\delta)$, and $p_{\delta}:=(z^*)^{-1}(q_{\delta})$. 
 Since $(z^*)^{-1}$ is Lipschitz on $[\delta, \infty)$, the chain rule for Sobolev functions allows us to compute $\nabla p_{\delta}=\frac{\chi_{\delta}(q)}{\rho}\nabla q$ where $\chi_{\delta}$ is the characteristic function $[\delta, \infty)$.  Furthermore, on the support of $\eta_{\delta}, \eta_{\delta}'$ it follows that $p=p_{\delta}$.  Hence,
 \begin{equation}\label{eq:test_gradient}
 \int_{Q_{\infty}} p\nabla\cdot (\vp\eta_{\delta}(q))=\int_{Q_{\infty}} p_{\delta}\rho \frac{\eta_{\delta}(q)}{\rho}\nabla\cdot \vp+p_{\delta}\rho\frac{\eta_{\delta}'(q)}{\rho}\nabla q\cdot \vp=\int_{Q_{\infty}}\frac{\eta_{\delta}(q)}{\rho}\nabla q\cdot  \vp
 \end{equation}
 Thus, $\nabla p$ is well defined as a distribution against any test vector field of the form $\eta_{\delta}(q)\psi$ where $\psi\in L^{\infty}_c([0,\infty);W^{1,\infty}_c(\RR^d))$ and when tested against these fields we have $\nabla p=\frac{1}{\rho}\nabla q$.  Examining equation (\ref{eq:test_gradient}), we see that we can in fact relax $\vp$ to belong to $L^2_c([0,\infty);L^2(\RR^d))$.  
 
 It is now clear that if $g$ is some function such that $0\leq g\leq \rho$ then we have $g\nabla p=\frac{g}{\rho}\nabla q$ on the support of $\eta_{\delta}(q)$. Since $\eta_{\delta}(q)\frac{g}{\rho}\nabla q\in L^2_{\loc}([0,\infty);L^2(\RR^d))$ independently of $\delta$, it follows that $g\nabla p\in L^2_{\loc}([0,\infty);L^2(\RR^d))$.  Thus, we can conclude that
 \[
 \frac{\rho_i}{\rho}\nabla q\cdot \vp=\rho_i\nabla p\cdot \vp
 \]
 where $\vp$ is any element of $L^2_c([0,\infty);L^2(\RR^d))$.  It now follows that $(\rho_1, \rho_2, p, n)$ is a solution to the system (\ref{eq:system}).
The regularity of $p$ can then be improved by arguing as in Propositions \ref{prop:edr} and \ref{prop:estimates}.
\end{proof}

The proofs of Theorems \ref{thm:existence_strict}  \ref{thm:existence_degenerate}, and \ref{thm:incompressible_limit} are now just corollaries of the previous proposition, Theorem \ref{thm:main}, and  Corollaries \ref{cor:existence_strict} and \ref{cor:existence_G4}. 

\appendix
\section{Some Convergence results for sequences of convex functions}
\begin{lemma}\label{lem:convex_convergence}
Let $f:\RR\to\RR\cup\{+\infty\}$ be a proper, lower semicontinuous, convex function such that $f^{-1}(\RR)$ is not a singleton.
If $f_k:\RR\to\RR\cup\{+\infty\}$ is a sequence of proper, lower semicontinuous convex functions such that $f_k$ converges pointwise everywhere to $f$ then the following properties hold:
\begin{enumerate}

    \item If $f$ is differentiable at a point $a\in \RR$, then \[
\limsup_{k\to\infty} \max\Big(| \sup \partial f_k(a) -f'(a)|\;, \; |\inf \partial f_k(a)-f'(a)|\Big)=0.
\]
\item The convergence of $f_k$ to $f$ is uniform on compact subsets of the interior of $f^{-1}(\RR)$.
\item $f_k^*$ converges pointwise everywhere to $f^*$ except possibly at the two exceptional values $b^+_{\infty}=\sup\{b\in\RR: z^*(b)<\infty\}, b^-_{\infty}=\inf\{b\in\RR: z^*(b)<\infty\}$.
\item If $f^*$ is differentiable at a point $b\in \RR$, then \[
\limsup_{k\to\infty} \max\Big(| \sup \partial f_k^*(b) -f^{*\, \prime}(b)|\;, \; |\inf \partial f_k^*(b)-f^{*\,\prime}(b)|\Big)=0,
\]
and the convergence of $f_k^*$ to $f^*$ is uniform on compact subsets of the interior of $(f^*)^{-1}(\RR)$.
\end{enumerate}
\end{lemma}
\begin{proof}
Let $a$ be a point of differentiability for $f$. Since $f'(a)$ exists and is finite,  there exists $\delta_0>0$ such that $f$ is finite on $[a-\delta_0, a+\delta_0]$. Fix some $\delta\in (0,\delta_0)$. The convergence of $f_k$ to $f$ implies that there must exist some $N, B$ sufficiently large such that $|f_k(a)|, |f_k(a-\delta)|, |f_k(a+\delta)|<B$ for all $k>N$. Now we can use convexity to bound
\[
\frac{f_k(a)-f_k(a-\delta)}{\delta}\leq \inf \partial f_k(a)\leq \sup \partial f_k(a)\leq \frac{f_k(a+\delta)-f_k(a)}{\delta}.
\]
Thus, 
\[
\limsup_{k\to\infty} \max\Big(| \sup \partial f_k(a) -f'(a)|\;, \; |\inf \partial f_k(a)-f'(a)|\Big)\leq |\frac{f(a)-f(a-\delta)}{\delta} -f'(a)|+|\frac{f(a+\delta)-f(a)}{\delta} -f'(a)|.
\]
Sending $\delta\to 0$ and using the fact that $f$ is differentiable at $a$, we get the desired result.

Now suppose that $[a_0, a_1]$ is an interval in the interior of $f^{-1}(\RR)$ and choose some $\delta>0$ such that $[a_0-\delta, a_1+\delta]$ is still in the interior of  $f^{-1}(\RR)$ and $f$ is differentiable at $a_0-\delta, a_1+\delta$.  Given any $a\in [a_0, a_1]$, we have 
\[
f'(a_0-\delta) \leq \inf \partial f(a)\leq \sup \partial f(a)\leq f'(a_1+\delta).
\]
It then follows from our above work that $\partial f_k(a)$ is uniformly bounded on $[a_0, a_1]$ for all $k$ sufficiently large.  Hence, $f_k$ is uniformly equicontinuous on $[a_0,a_1]$ and thus converges uniformly to $f$.

Now we consider $f^*$.  Fix some $b\in \RR$. If $f^*(b)=+\infty$, then for each $j\in \ZZ_+$ there exists $a_j\in \RR$ such that $a_jb-f(a_j)>j$.  We can then compute
\[
\liminf_{k\to\infty} f_k^*(b)\geq \liminf_{k\to\infty} a_jb-f_k(a_j)>j.
\]
Thus, $\liminf_{k\to\infty} f_k^*(b)=+\infty$.

If $b_{\infty}^-=b_{\infty}^+$ then we are already done.  Otherwise, given $b\in (b_{\infty}^-, b_{\infty}^+)$, let $a_0, a_1$ be the infimum and supremum of the set $\{a\in\RR: b\in \partial f(a)\}$ respectively. Since $b\in (b_{\infty}^-, b_{\infty}^+)$, $a_0, a_1$ must exist and are finite. Furthermore, for any $a\in [a_0, a_1]$ we have $f^*(b)=ab-f(a)$. 

If we fix some $\delta>0$, it follows that
$\frac{f(a_0)-f(a_0-\delta)}{\delta} < b< \frac{f(a_1+\delta)-f(a_1)}{\delta}$,
and hence for all $k$ sufficiently large 
$\frac{f_k(a_0)-f_k(a_0-\delta)}{\delta} < b< \frac{f_k(a_1+\delta)-f_k(a_1)}{\delta}$
Hence, for all $k$ sufficiently large
\[
f_k^*(b)= \sup_{a\in [a_0-\delta,a_1+\delta]} ab-f_k(a).
\]
It is now clear that $\liminf_{k\to\infty} f_k^*(b)\geq f^*(b)$.  

If $a_0<a_1$ then $f$ is differentiable at all $a\in (a_0, a_1)$ and $f'(a)=b$.  Therefore if we fix some $a'\in (a_0, a_1)$ and for each $k$ choose some $b_k\in \partial f_k(a')$ then
\[
f_k^*(b)\leq \sup_{a\in [a_0-\delta, a_1+\delta]} ab-f_k(a')-b_k(a-a')\leq \max\Big((a_0-\delta)b-b_k(a_0-\delta-a')\, , \, (a_1+\delta)b-b_k(a_1+\delta-a') \Big)-f_k(a')
\]
Since $b_k\to b$, we get
\[
\limsup_{k\to\infty} f_k^*(b)\leq a'b-f(a')=f^*(b).
\]

Otherwise if $a_0=a_1$, then since $f^{-1}(\RR)$ is not a singleton, we can find a sequence $a_j$ converging to $a_0$ such that $f$ is differentiable at $a_j$ for all $j$.  For each $j,k$ choose some $b_{j,k}\in \partial f_k(a_j)$ and note that $\lim_{k\to\infty} b_{j,k}=f'(a_j)$. Thus, we can compute 
\[
\limsup_{k\to\infty}f_k^*(b)\leq \limsup_{k\to\infty}\sup_{a\in [a_0-\delta, a_0+\delta]} ab-f_k(a_j)-b_{j,k}(a-a_j)
\]
\[
\leq a_0b-f(a_j)-f'(a_j)(a_0-a_j)+\delta(|b|+|f'(a_j)|)
\]
Sending $\delta\to 0$ and then $j\to\infty$, it follows that
\[
\limsup_{k\to\infty} f^*_k(b)\leq a_0b-f(a_0)=f^*(b)
\]
as desired.

We have now shown that $\lim_{k\to\infty} f^*_k(b)=f^*(b)$ except possibly at $b=b_{\infty}^+, b_{\infty}^-$. Since $b_{\infty}^+, b_{\infty}^-$ does not lie in the interior of $(f^*)^{-1}(\RR)$, we can use the same argument we used to establish properties (1) and (2) to establish property (4).
\end{proof}

\begin{lemma}\label{lem:convex_convergence_2}
Let $z:\RR\to\RR\cup\{+\infty\}$ be an energy satisfying (z1-z3) and let $z_k:\RR\to\RR\cup\{+\infty\}$ be a sequence of energies satisfying (z1-z3)  such that $z_k$ converges pointwise everywhere to $z$. If we set $b_{\infty}=\inf\{b\in \RR: z^*(b)=+\infty\}$
 then $(z_k^*)^{-1}$ converges uniformly to $(z^*)^{-1}$ on compact subsets of $\big(0,z^*(b_{\infty})\big)$.  
\end{lemma}
\begin{proof}
If $z^*(b_{\infty})=0$, then there is nothing to prove.  Otherwise, given $\epsilon\in (0, z^*(b_{\infty}))$ there must exist $b_{\epsilon/2}<b_{\epsilon}\in \RR$ such that $z^*(b_{\epsilon/2})=\epsilon/2$ and $z^*(b_{\epsilon})=\epsilon.$ It then follows that for all $b\geq b_{\epsilon}$ and $k$ sufficiently large
\[
\frac{\epsilon}{4(b_{\epsilon}-b_{\epsilon/2})}\leq \inf \partial z^*_k(b).
\]
 As a result, $(z_k^*)^{-1}$ is uniformly Lipschitz on $[\epsilon, z^*(b_{\infty}))$.   Choose some value $a\in [\epsilon, z^*(b_{\infty}))$ and let $\bar{b}=(z^*)^{-1}(a)$.   Let $a_k=z^*_k(\bar{b})$ and note that once $k$ is sufficiently large we must have $a\in z_k^*(\RR)$.  Thus,
\[
|(z^*)^{-1}(a)-(z^*_k)^{-1}(a)|=|\bar{b}-(z^*_k)^{-1}(a_k+a-a_k)|\leq L_{\epsilon}|a-a_k|=L_{\epsilon}|z^*(\bar{b})-z^*_k(\bar{b})|
\]
Now the uniform convergence of $z^*_k$ to $z^*$ on compact subsets of $(-\infty,b_{\infty})$ combined with the Lipschitz bound implies the uniform convergence of $(z_k^*)^{-1}$ to $(z^*)^{-1}$ on compact subsets of $(0,z^*(b_{\infty}))$.

\end{proof}

\begin{lemma}\label{lem:in_measure}
Suppose that $f_k:\RR\to \RR$ is a sequence of proper, lower semicontinuous, convex functions that converge pointwise everywhere to a function $f:\RR\to\RR$ that is also proper, lower semicontinuous, and convex with $a_0:=\inf\{a\in \RR: f(a)<\infty\}<\sup\{a\in \RR: f(a)<\infty\}=:a_1$.
 Let $u_k\in L^1_{\loc}(Q_{\infty})$ be a sequence of uniformly integrable functions such that  $u_k$ converges weakly in $L^1_{\loc}(Q_{\infty})$ to a limit $u\in L^1_{\loc}(Q_{\infty})$. Suppose in addition that the sequence $f_k(u_k)$ converges weakly in $L^1_{\loc}(Q_{\infty})$ to $f(u)\in L^1_{\loc}([0,\infty);L^1(\RR^d))$.  If there exists $v\in L^{\infty}_{\loc}(Q_{\infty})$ such that $v\in \partial f(u)$ and $f$ is strictly convex on the interior of $f^{-1}(\RR)$, then $u_k$ converges locally in measure to $u$.
\end{lemma}
\begin{proof}
Fix a compact set $D\subset Q_{\infty}$.
Fix $\epsilon>0$ and let $S_{k,\epsilon}=\{(t,x)\in D: u_k>a_1+\epsilon\}$. Choose some value $a\in (a_0, a_1)$.  Since $f(a)$ is finite and $f_k(a_1+\epsilon)$ must approach $\infty$ as $k\to\infty$, it follows that $f_k$ is increasing at $a_1+\epsilon$ for all $k$ sufficiently large.  Therefore, 
\[
\limsup_{k\to\infty} |S_{k,\epsilon}|f_k(a_1+\epsilon)\leq \limsup_{k\to\infty}\int_{S_{k,\epsilon}} f_k(u_k)\leq \limsup_{k\to\infty}\int_D |f_k(u_k)|<\infty,
\]
where in the last inequality we used the fact that the sequence $f_k(u_k)$ is uniformly bounded in $L^1_{\loc}(Q_{\infty})$. Of course the above inequality is only possible if $\limsup_{k\to\infty}|S_{k,\epsilon}|=0$. A similar argument shows that the measure of the sets $\{(t,x)\in D: u_k(t,x)<a_0-\epsilon\}$ also vanishes in the $k\to\infty$ limit.

Given some $\delta<(a_1-a_0)/2$,  define $u_{k,\delta}=\max(a_0+\delta, \min(u_k,a_1-\delta))$ and  $v_{k,\delta}=\inf \partial f_k(u_{k,\delta})$.  From the convergence properties of $u_k$ and $f_k(u_k)$  we have 
\[
\lim_{k\to\infty}\int_{D} f_k(u_k)-f(u)-v(u_k-u)=0.
\]
Therefore,
\[
0\geq \limsup_{k\to\infty} \int_{D} f_k(u_{k,\delta})+v_{k,\delta}(u_k-u_{k,\delta})-f(u)-v(u_k-u).
\]
For $\delta>0$ fixed, Lemma \ref{lem:convex_convergence} implies that $f_k(u_{k,\delta})-f(u_{k,\delta})$ converges uniformly to zero.  Hence, 
\[
0\geq \limsup_{k\to\infty} \int_{D} v_{k,\delta}(u_k-u_{k,\delta})+f(u_{k,\delta})-f(u)-v(u_k-u).
\]
For $\delta$ sufficiently small, either $v_{k,\delta}(u_k-u_{k,\delta})$ is positive or $v_{k,\delta}$ is bounded.  Either way, from the uniform integrability of $u_k$ and our work in the first paragraph, it follows that  
\[
\lim_{\delta\to 0}\limsup_{k\to\infty} \int_D v_{k,\delta}(u_k-u_{k,\delta})+v(u_k-u_{k,\delta})\geq 0.
\]
Thus,
\begin{equation}\label{eq:bregman}
0\geq \lim_{\delta\to 0}\limsup_{k\to\infty} \int_{D} f(u_{k,\delta})-f(u)-v(u_{k,\delta}-u).
\end{equation}

Given $\epsilon>0$, let $D_{k,\delta, \epsilon}=\{(t,x)\in D: |u_{k,\delta}-u|>\epsilon\}$. Equation (\ref{eq:bregman}) is a Bregman divergence of a strictly convex function, therefore,
\[
\lim_{\delta\to 0}\limsup_{k\to\infty} |D_{k,\delta,\epsilon}|=0.
\]
If we let $u'_k=\max(\min(a_1, u_k),a_0)$ and $D'_{k,\epsilon}=\{(t,x)\in D: |u_{k}'-u|>\epsilon\}$ then it is clear that $\limsup_{k\to\infty}|D_{k,\epsilon}|=0.$ Thus, $u_k'$ converges locally in measure to $u$.  From our work in the first paragraph we know that $u_k-u_k'$ converges locally in measure to zero, thus we are done.

\end{proof}

\bibliographystyle{alpha}
\bibliography{references}

\newcommand{\etalchar}[1]{$^{#1}$}
\begin{thebibliography}{BKMP03}

\bibitem[AKY14]{alexander_kim_yao}
Damon Alexander, Inwon Kim, and Yao Yao.
\newblock Quasi-static evolution and congested crowd transport.
\newblock {\em Nonlinearity}, 27(4):823--858, mar 2014.

\bibitem[BCP20]{price_xu}
Xiangsheng~Xu Brock C.~Price.
\newblock Global existence theorem for a model governing the motion of two cell
  populations.
\newblock {\em Kinetic \& Related Models}, 13(6):1175--1191, 2020.

\bibitem[BKMP03]{BYRNE2003567}
H.M. Byrne, J.R. King, D.L.S. McElwain, and L.~Preziosi.
\newblock A two-phase model of solid tumour growth.
\newblock {\em Applied Mathematics Letters}, 16(4):567--573, 2003.

\bibitem[BM14]{santambrogio_dendritic}
Filippo~Santambrogio Bertrand~Maury, Aude Roudneff-Chupin.
\newblock Congestion-driven dendritic growth.
\newblock {\em Discrete \& Continuous Dynamical Systems}, 34(4):1575--1604,
  2014.

\bibitem[BPPS19]{bubba_perthame}
Federica Bubba, Benoît Perthame, Camille Pouchol, and Markus Schmidtchen.
\newblock Hele–shaw limit for a system of two reaction-(cross-)diffusion
  equations for living tissues.
\newblock {\em Archive for Rational Mechanics and Analysis}, 236(2):735–766,
  Dec 2019.

\bibitem[CFSS18]{carrillo_1d}
J.~A. Carrillo, S.~Fagioli, F.~Santambrogio, and M.~Schmidtchen.
\newblock Splitting schemes and segregation in reaction cross-diffusion
  systems.
\newblock {\em SIAM Journal on Mathematical Analysis}, 50(5):5695--5718, 2018.

\bibitem[GP{\'S}G19]{gwiazda_perthame}
Piotr Gwiazda, Beno{\^i}t Perthame, and Agnieszka {\'S}wierczewska-Gwiazda.
\newblock A two-species hyperbolic–parabolic model of tissue growth.
\newblock {\em Communications in Partial Differential Equations},
  44(12):1605--1618, 2019.

\bibitem[JKT21]{Jacobs2021}
Matt Jacobs, Inwon Kim, and Jiajun Tong.
\newblock Darcy's law with a source term.
\newblock {\em Archive for Rational Mechanics and Analysis}, 239(3):1349--1393,
  Mar 2021.

\bibitem[KM18]{Kim_Meszaros}
Inwon Kim and Alp{\'a}r~Rich{\'a}rd M{\'e}sz{\'a}ros.
\newblock On nonlinear cross-diffusion systems: an optimal transport approach.
\newblock {\em Calculus of Variations and Partial Differential Equations},
  57(3):79, Apr 2018.

\bibitem[KT20]{kim_tong}
Inwon Kim and Jiajun Tong.
\newblock Interface dynamics in a two-phase tumor growth model, 2020.

\bibitem[MRCS10]{santambrogio_crowd_motion}
Bertrand Maury, Aude Roudneff-Chupin, and Filippo Santambrogio.
\newblock A macroscopic crowd motion model of gradient flow type, 2010.

\bibitem[Ott01]{otto2001pme}
Felix Otto.
\newblock The geometry of dissipative evolution equations: the porous medium
  equation.
\newblock {\em Comm. Partial Differential Equations}, 26(1-2):101--174, 2001.

\bibitem[PQV14]{pqv2014}
Beno{\^i}t Perthame, Fernando Quir{\'o}s, and Juan~Luis V{\'a}zquez.
\newblock The hele--shaw asymptotics for mechanical models of tumor growth.
\newblock {\em Archive for Rational Mechanics and Analysis}, 212(1):93--127,
  Apr 2014.

\bibitem[PT08]{Preziosi2008}
Luigi Preziosi and Andrea Tosin.
\newblock Multiphase modelling of tumour growth and extracellular matrix
  interaction: mathematical tools and applications.
\newblock {\em Journal of Mathematical Biology}, 58(4-5):625--656, October
  2008.

\bibitem[PV15]{perthame_2015}
Beno{\^i}t Perthame and Nicolas Vauchelet.
\newblock Incompressible limit of a mechanical model of tumour growth with
  viscosity.
\newblock {\em Philosophical transactions. Series A, Mathematical, physical,
  and engineering sciences}, 373(2050):20140283, Sep 2015.
\newblock 26261366[pmid].

\bibitem[RBE{\etalchar{+}}10]{Ranft20863}
Jonas Ranft, Markus Basan, Jens Elgeti, Jean-Fran{\c c}ois Joanny, Jacques
  Prost, and Frank J{\"u}licher.
\newblock Fluidization of tissues by cell division and apoptosis.
\newblock {\em Proceedings of the National Academy of Sciences},
  107(49):20863--20868, 2010.

\end{thebibliography}
\end{document}